\RequirePackage{fix-cm}

\documentclass[smallextended]{svjour3}       

\smartqed  

\usepackage{graphicx}
\usepackage{appendix}
\usepackage[dvipdfmx]{hyperref}
\hypersetup{colorlinks = true, 
	linkcolor = blue, 
	urlcolor = blue,
	citecolor = blue}
\usepackage{srcltx,bm}
\usepackage{indentfirst}
\usepackage{framed}
\usepackage{epsfig}
\usepackage{psfrag}
\usepackage{overpic}
\usepackage{multirow}
\usepackage{longtable}
\usepackage[hang]{subfigure}

\usepackage{amsthm}
\usepackage{amsmath,amssymb}
\usepackage[dvips]{color}
\usepackage[mathscr]{euscript}
\usepackage{xcolor}
\usepackage{supertabular}
\usepackage{colortbl,booktabs}
\usepackage{mathdots}
\usepackage{appendix}
\usepackage{txfonts}
\usepackage{ragged2e}
\usepackage{caption}
\newcommand{\figref}[1]{{\bf Fig.~\ref{#1}}}
\newcommand\bbR{\mathbb{R}}
\newcommand\bV{\boldsymbol{V}}
\newcommand\lv{\left\vert}
\newcommand\rv{\right\vert}
\newcommand\dd{\mathrm{d}}
\newcommand\de{\mathrm{e}}
\newcommand\ri{{\rm{i}}} 
\newcommand\eps{\varepsilon}

\spnewtheorem{algorithm}{Algorithm}{\bf}{\rm}

\allowdisplaybreaks

\begin{document}

\title{High Order Finite Difference Schemes for the Transparent Boundary
	Conditions and Their Applications in the 1D Schr{\"o}dinger-Poisson
	Problem}

\titlerunning{High Order FD Schemes for the TBCs of the 1D Schr{\"o}dinger-Poisson
	Problem}        

\author{Meili Guo         \and
        Haiyan Jiang      \and
        Tiao Lu           \and
        Wenqi Yao$^{\ast}$\thanks{*Corresponding author}
        }

\institute{Meili Guo \at
              School of Mathematics and Statistics, Beijing Institute of
              Technology, 100081, Beijing, China\\
              \email{guomeili@bit.edu.cn}           
           \and
         Haiyan Jiang \at
              School of Mathematics and Statistics, Beijing Institute of
              Technology, 100081, Beijing, China\\
              \email{hyjiang@bit.edu.cn}
           \and
           Tiao Lu  \at
           	   HEDPS, CAPT, LMAM, School of Mathematical Sciences,
           	  Peking University, 100871, Beijing, China\\
           	  \email{tlu@pku.edu.cn}
           \and
           Wenqi Yao \at
           School of Mathematics, South China University of
           Technology, 510641, Guangzhou, Guangdong, China\\
           \email{yaowq@scut.edu.cn}
           }

\maketitle

\begin{abstract}
The 1D Schr{\"o}dinger equation closed with the
transparent boundary conditions(TBCs) is known as a
successful model for describing quantum effects, and is usually
considered with a self-consistent Poisson equation
in simulating quantum devices.
We introduce discrete fourth order transparent boundary conditions(D4TBCs),
which have been proven to be essentially
non-oscillating when the potential vanishes, and
to share the same accuracy order with
the finite difference scheme used to
discretize the 1D Schr{\"o}dinger equation.
Furthermore, a framework of analytic discretization of
TBCs(aDTBCs) is proposed, which does
not introduce any discretization error, thus
is accurate.
With the accurate discretizations, one is able to
improve the accuracy of the discretization for the 1D Schr{\" o}dinger 
problem to arbitrarily high levels. 
As numerical tools, two globally fourth order compact finite
difference schemes are proposed for the 1D
Schr{\"o}dinger-Poisson problem, involving either of
the D4TBCs or the aDTBCs, respectively, and the uniqueness of solutions of both
discrete Schr{\"o}dinger problems are rigorously proved.
Numerical experiments, including simulations of a resistor and two 
nanoscale resonant tunneling diodes, verify the accuracy order of
the discretization schemes and show potential of
the numerical algorithm introduced for
the 1D Schr{\"o}dinger-Poisson problem in
simulating various quantum devices.

\end{abstract}

\section{Introduction}\label{Introduction}
The present trend of semiconductor fabrication
technology shows a dramatic size reduction of semiconductor
devices.
As a result, quantum transport models are considered
in order to clearly understand the electrical properties
of nanoscale devices.
Three equivalent approaches are proposed
to model nanoscale devices, including the Wigner
transport equation\cite{1991Numerical,frensley1987},
the non-equilibrium Green's function(NEGF) method
\cite{2000Nanoscale,2008Boundary}, and the Schr{\"o}dinger
equation\cite{2003nanoMOS,Rana1996Self,2013Transient}.

We use the 1D Schr{\"o}dinger equation to describe the
characteristics of quantum devices, where
appropriate boundary conditions are needed
to confine the problem in a bounded domain.
The set of transparent boundary conditions (TBCs)\cite{ben1997one}
is a good candidate, and the yielded 1D Schr{\"o}dinger problem,
i.e., the 1D Schr{\"o}dinger equation equipped with the
TBCs, admits a unique solution, thus is well-defined.
When simulating realistic devices, one has to consider
self-consistent potential due to the difference of
electron density and doping density.
Therefore, a nonlinear 1D Schr{\"o}dinger-Poisson problem
is yielded naturally, by taking into account a
potential governed by a Poisson equation equipped with
appropriate boundary conditions, such as the electric
neutrality boundary conditions used in this paper.

Discretizing TBCs is crucial in the
development of numerical tools for solving
the 1D Schr{\"o}dinger problem.
As is described in \cite{Anton2001},
inappropriate discretizations of
TBCs lead to unphysical spurious oscillations
in numerical solutions when potential vanishes.
To solve this problem, a set of discrete fourth order
transparent boundary conditions(D4TBCs) are introduced
in this paper together with a fourth order compact
finite difference discretization of the 1D Schr{\"o}dinger
equation.
In addition, we prove the solution
of the yielded discretization system corresponding to
the 1D Schr{\"o}dinger problem is unique with the form
of discrete plain waves, which in turn explains
extinction of spurious oscillations when potential
vanishes.
However, discrete transparent boundary conditions(DTBCs),
no matter our D4TBCs or Arnold's\cite{Anton2001},
have drawbacks.
To deduce discrete transparent boundary conditions,
a discrete dispersion relation is needed,
to obtain which an algebraic equation has to be solved.
The order of the algebraic equation is determined
by the stencil used for spatial discretization of the
1D Schr{\"o}dinger equation,
and the larger the stencil is, the higher
the order of the algebraic equation will be.
Furthermore, DTBCs introduce discretization error.
To overcome drawbacks possessed by DTBCs, we further propose
the framework of analytic discrete transparent boundary conditions(aDTBCs),
which is accurate in discretizing the TBCs.
More importantly, without high order algebraic equations 
to solve, we are able to improve the accuracy order of 
the discretization of the 1D Schr{\" o}dinger problem to arbitrarily high,
by using aDTBCs. 
For numerical simulations, we introduce a fourth order compact finite 
difference discretization of the 1D Schr{\"o}dinger problem
involving the aDTBCs, which also well-defined due to the
uniqueness of solution.

The resonant tunneling diode(RTD) has been widely used
as a high-frequency and low-consumption oscillator
or switch, and is studied both experimentally
and theoretically over decades\cite{datta1990quantum,1974Resonant}.
In order to simulate RTDs, we propose an
algorithm to solve the 1D Schr{\"o}dinger-Poisson problem,
where two discrete models are available, by equipping
D4TBCs and aDTBCs, respectively.
Numerical experiments are carried out following the
algorithm and include two stages.
In the first stage, we simulate a classic $n^{++}-n^+-n^{++}$
resistor and an RTD, with the potential known as a {\it priori},
to verify fourth order accuracy of both
discrete models.
In the second stage, the algorithm is applied to
simulate another RTD, with the potential updated
simultaneously.
I-V characteristic curves of both RTDs are simulated,
where typical features of RTDs are observed, including
negative differential resistance
\cite{1985Room,tsu1973tunneling},
and corresponding performance in electron density,
potential and transmission coefficients.

This article is organized as follows.
In  Sect.~\ref{sec:schrodinger_poisson_problem}, 
we make a brief review about 
the 1D Schr{\"o}dinger-Poisson problem.
Finite difference discretizations of the
Schr{\"o}dinger equation
and TBCs are thoroughly discussed in  Sect.~\ref{sec:discretization_scheme},
where well-posedness and accuracy of yielded discrete systems
are proved.
Discretization of the self-consistent Poisson equation
equipped with electric neutrality boundary conditions
is given in  Sect.~\ref{sec:Algorithm}, and an algorithm for solving
the discrete Schr{\"o}dinger-Poisson problem is
introduced subsequently.
 Sect.~\ref{sec:numerical_experiments} is devoted to
simulations of a short $n^{++}-n^+-n^{++}$ resistor
and two resonant tunneling diodes by using the
algorithm presented in  Sect.~\ref{sec:Algorithm}.

\section{1D  Schr{\"o}dinger-Poisson problem}\label{sec:schrodinger_poisson_problem}
We consider 1D Schr{\"o}dinger equation on the real line:
\begin{equation}\label{eq:schrodinger}
	\left(-\frac{\hbar^2}{2m^*} \dfrac{\dd^2}{\dd x^2} + V(x)\right)
	\psi(x;k)
	= E\psi(x;k), \quad x\in \bbR,\quad
	\forall k\in \bbR,
\end{equation}
where $\psi(x;k)$ denotes the complex wave
function of the electron, $E=\frac{\hbar^2k^2}{2m^*}$
is the energy of an electron at steady state,
physical parameters $\hbar$ and $m^*$ are the reduced
Planck constant and the effective mass of the electron.
The potential function $V(x)=V_{b}(x) + V_{s}(x)$,
where $V_{b}(x)$ is the conduction band
structure, and $V_{s}(x)$ is an external potential exerted
to the device.
$V_s(x)$ is either a given potential function
or determined by a Poisson equation
equipped with Neumann boundary conditions,
which reads
\begin{eqnarray}\label{eq:poisson}
	-\dfrac{\dd^2}{\dd x^2}V_{s}(x) = \frac{q_e^2}{\varepsilon}(n(x)-N_d(x)), \quad
	x\in[0,L],\\\label{poisson_bc}
	V_{s}^\prime(0) = 0, \quad V_{s}^\prime(L) =  0,
\end{eqnarray}
where $\varepsilon$ is the dielectric constant,
$q_e$ is the electronic charge, $n(x)$ and $N_d(x)$
denote electron density and
doping density, respectively.
Furthermore, $V$ is extended to $x<0$ and $x>L$ by
setting
\begin{equation}\label{V_extend}
	V(x)= \left\{\begin{array}{lr}V(0) ,
		& x < 0,\\ V(L), & x > L,
	\end{array}\right.\end{equation}
in order to meet the requirement of \eqref{eq:schrodinger}.
When \eqref{V_extend} is applied,
\eqref{eq:schrodinger}
admits solution which takes the form of:
\begin{equation}\label{wave_outside_boundary_form1}
	\psi(x;k) = \left\{ \begin{array}{lr}
		A \de^{-\mathrm{i}k_1 x} + B \de^{\mathrm{i}k_1x},
		& x\leq 0, \\
		C \de^{-\mathrm{i}k_2 x} + D \de^{\mathrm{i}k_2 x}, &
		x\geq L,
	\end{array}\right.
\end{equation}
where $A$, $B$, $C$ and $D$ are arbitrary constants, and $k_1 = \sqrt{\frac{2m^* (E-V(0))}{\hbar^2}}$,
$k_2=\sqrt{\frac{2m^* (E-V(L))}{\hbar^2}}$, where $E\geq \max\{V(0),V(L)\}$ is mandatory.
Let the wave $\psi(x;k)$ enter the device region,
i.e., $x\in [0,L]$, from $x\leq 0$ with amplitude $1$, and
let a part
of the wave be reflected at $x= 0$ while
the remaining part be transmitted and travel to $+\infty$.
Based on the above assumption,
\eqref{wave_outside_boundary_form1} is further
rewritten as
\begin{equation}\label{wave_outside_boundary}
	\psi(x;k) = \left\{ \begin{array}{lr}
		r \de^{-\mathrm{i}k_1 x} + \de^{\mathrm{i}k_1x}, &
		x\leq 0, \\
		t \de^{\mathrm{i}k_2 x}, &
		x\geq L,
	\end{array}\right.
\end{equation}
where $r$ and $t$ are the so called reflection and
transmission coefficients satisfying $\lv r\rv^2
+\lv t\rv^2 = 1$.
\eqref{eq:schrodinger} and \eqref{eq:poisson} would
not form
a well-posed problem confined to $x\in[0,L]$, unless
appropriate boundary conditions are equipped to
\eqref{eq:schrodinger} restricted on $x\in [0,L]$.
As discussed in \cite{ben1997one}, artificial transparent
boundary conditions(TBCs) are appended to close
\eqref{eq:schrodinger}, which read
\begin{eqnarray}\label{eq:TBC_left}
	\psi'(0;k)+\mathrm{i} k_1\psi(0;k) &=& 2\mathrm{i} k_1,\\
	\label{eq:TBC_right}
	\psi'(L;k) - \mathrm{i} k_2 \psi(L;k) &=& 0.
\end{eqnarray}
One thing is noticed that
TBCs are necessary conditions of \eqref{wave_outside_boundary},
without requiring
the {\it priori} values of $R$ and $T$.
At equilibrium, the device is connected
to two contacts with the same Fermi levels, and
the electron density $n(x)$ in \eqref{eq:poisson}
is related to the wave function through

\begin{equation}\label{density}
	n(x) =
	\dfrac{1}{2\pi}\int_{-\infty}^\infty  F\left(E_F-E(k)\right)\lv\psi(x;k)\rv^2 \dd k,
\end{equation}
where
\begin{equation}\label{fermi}
	F\left(E_F-E\right) = \frac{m^*k_BT_L}{\pi\hbar^2}\ln\left(
	1+ \exp\left(\frac{E_F - E}{k_BT_L}\right)\right),
\end{equation}
and $E_F$ denotes the Fermi level of the left contact,
$k_B$ and $T_L$ are Boltzmann constant and temperature
of lattice, respectively.
However, when a bias voltage $V_{ds}$ is applied to one
of two contacts, saying the right one,
then the total electron density is given by\cite{DATTA1998771}

\begin{equation}\label{total_density}
	n(x)  =\dfrac{1}{2\pi}\left[ \displaystyle\int_{0}^{\infty}F\left(E_F-E(k)\right)
	\lv\psi(x;k)\rv^2\dd k + \displaystyle\int_{-\infty }^{0}
	F\left(E_F-q_eV_{ds}-E(k)\right)\lv\psi(x;k)\rv^2\dd k\right].
\end{equation}
Under this circumstance, the current density reads
\begin{equation}\label{current_density}
	I = \dfrac{q_e}{2\pi\hbar}\int_0^\infty T(E) \left[F\left(E_F-E\right)
	-F\left(E_F-q_eV_{ds}-E\right)\right]\dd E,
\end{equation}
where the transmission coefficient $T(E)$ is defined as
the ratio of the transmitted current $I_{\rm trans}$
and the incident current $I_{\rm inc}$, and can be expressed as
\begin{equation}\label{transmission_coefficient}
	T(E) = \dfrac{I_{\rm trans}}{I_{\rm inc}}=1 - \lv r\rv^2.
\end{equation}

\section{Discretization schemes for the
	1D Schr{\"o}dinger problem}\label{sec:discretization_scheme}
The whole problem contains two subproblems, i.e.,
the 1D Schr{\"o}dinger problem(\eqref{eq:schrodinger}
equipped with TBCs),
and the self-consistent Poisson problem(\eqref{eq:poisson}
-\eqref{poisson_bc}),
and they both correspond to one model problem:
\begin{eqnarray}\label{model_problem}
	u_{xx} &=& f(x), \quad x\in [0,L],\\\label{model_BC1}
	u_x(0)+a_l u(0) &=& b_l, \\\label{model_BC2}
	u_x(L)+a_r u(L) &=& b_r,
\end{eqnarray}
where $u,f,a_{l,r},b_{l,r}\in \mathbb{C}$ depend
on the specific problem.
In this section, we firstly propose a globally fourth order
compact scheme for solving the model problem
\eqref{model_problem}, which could be
directly applied to discretize the Schr{\"o}dinger
equation and the Poisson equation.
In the rest of this section,
lots of efforts are devoted to propose and discuss
two optimal discretizations of the TBCs, i.e.,
the D4TBCs and the aDTBCs.

\subsection{A globally fourth order compact scheme for the model problem}
\label{subsec:model+discretization}
Set a uniform mesh, saying $x_i = i \Delta x$,
$i=0,1,\cdots N_x$, $N_x = \frac{L}{\Delta x}$.
Following one of the traditional ways of constructing
compact finite difference schemes\cite{CompactScheme,CompactSchemeShu1},
we approximate the first and second derivatives
of any smooth function $u(x)$
via the linear combination of nodal values:
\begin{equation}\label{Compact_Scheme_general_first_derivative}
	\sum_{k=-p}^p a_k\left(u_{x}\right)_{j+k}
	= \frac{1}{\Delta x}\sum_{k=1}^q b_k\left(
	u_{j+k} - u_{j-k}\right),
\end{equation}
\begin{equation}\label{Compact_Scheme_general_second_derivative}
	\sum_{k=-l}^l \alpha_k\left(u_{xx}\right)_{j+k}
	= \frac{1}{\Delta x^2}\sum_{k=1}^s \beta_k\left(
	u_{j+k} - 2u_j + u_{j-k}\right),
\end{equation}
where $u_i$, $\left(u_x\right)_i$ and $\left(u_{xx}\right)_i$ represent
approximations of $u$, $u_x$ and $u_{xx}$ at
$x_i$, respectively.
Let $p=q=l=s=1$ and match Taylor series coefficients
of both sides of \eqref{Compact_Scheme_general_first_derivative} and
\eqref{Compact_Scheme_general_second_derivative}
at $x_i$ to fourth order, respectively.
As a result, one gets fourth order compact schemes
for discretizing $u_x$ and $u_{xx}$ at $x_i$,
$i=1,2,\cdots,N_x-1$, which read
\begin{equation}\label{Compact_4th_general_first_derivative}
	\frac{1}{4} \left(u_{x}\right)_{i-1}
	+ \left(u_{x}\right)_i
	+ \frac{1}{4} \left(u_{x}\right)_{i+1}
	= \frac{3}{4\Delta x} \left( u_{i+1}
	- u_{i-1}\right),
\end{equation}
and
\begin{equation}\label{Compact_4th_general_second_derivative}
	\frac{1}{10} \left(u_{xx}\right)_{i-1}
	+ \left(u_{xx}\right)_i
	+ \frac{1}{10} \left(u_{xx}\right)_{i+1}
	= \frac{6}{5\Delta x^2} \left( u_{i+1} -
	2u_i + u_{i-1}\right).
\end{equation}

Recalling that $u$ solves \eqref{model_problem},
we replace $\left(u_{xx}\right)_{i\pm 1}$
\eqref{Compact_4th_general_second_derivative} with
$f_{i\pm 1}$, where $f_i$ represents the approximation
of $f(x_i)$(with at least fourth order accuracy),
and obtain a fourth order discretization scheme
for \eqref{model_problem}, which reads
\begin{equation}\label{model_equation_compact}
	\left(\lambda u_{j-1}-f_{j-1}\right)
	-\left(2\lambda u_j+10 f_j\right)
	+ \left(\lambda u_{j+1} - f_{j+1}\right) = 0,\quad
	j=1\cdots,N_x-1,
\end{equation}
where $\lambda = \frac{12}{\Delta x^2}$.
To discretize \eqref{model_BC1}-\eqref{model_BC2},
two ghost points: $x_{-1}\triangleq x_0-\Delta x$
and $x_{N_x+1} \triangleq x_{N_x}+\Delta x$ are introduced.
We further adopt Taylor series expansion
of $u_x(x_{i-1})$ and $u_x(x_{i+1})$ at $x_i$
up to fourth order in
\eqref{Compact_4th_general_first_derivative},
and obtain
\begin{equation}\label{Compact_4th_general_first_derivative_second_form}
	\frac{3}{2} \left(u_x\right)_i +
	\frac{\Delta x^2}{4} \left( u_{xxx}\right)_i
	=  \frac{3}{2} \left(u_x\right)_i +
	\frac{\Delta x^2}{4} \left( f_x\right)_i
	= \frac{3}{4\Delta x}\left(u_{i+1}-u_{i-1}\right),
\end{equation}
where \eqref{model_problem} is used.
Therefore, we use second order central finite
difference scheme to replace $\left(f_x\right)_i$
in \eqref{Compact_4th_general_first_derivative_second_form}
and obtain an alternative fourth order discretization of
$u_x(x_i)$, which reads
\begin{equation}\label{Compact_4th_general_first_derivative_third_form}
	\left(u_x\right)_i =
	-\frac{\Delta x}{12}\left(f_{i+1}-f_{i-1}\right)
	+ \frac{1}{2\Delta x}\left(u_{i+1}-u_{i-1}\right).
\end{equation}
We will discuss the usage of
\eqref{Compact_4th_general_first_derivative_third_form}
in discretizing boundary conditions of Schr{\"o}dinger
equation and Poisson equation in detail, respectively,
in subsequent sections.

\subsection{The D4TBCs}
\label{subsec:essentially_non_oscillating}
We instantiate \eqref{model_problem} with the Schr{\"o}dinger
problem.
Applying \eqref{model_equation_compact},
we discretize \eqref{eq:schrodinger} at interior grid nodes as
\begin{equation}\label{4th_compact_schrodinger_interior}
	\left(\lambda \psi_{j-1}- f_{j-1}\right)
	-\left(2\lambda \psi_j+10 f_j\right)
	+ \left(\lambda \psi_{j+1} - f_{j+1}\right) = 0,\quad
	j=1,2\cdots,N_x-1.
\end{equation}
where $f_j=\frac{2m^*}{\hbar^2}(V_j-E)\psi_j$,
$V_j$ and $\psi_j$ denote by the approximations
of $V(x_j)$ and $\psi(x_j;k)$, respectively.
Furthermore, $\psi_x(0;k)$ is approximated
with $\left(\psi_x\right)_0$, where
\begin{equation}\label{phi_x_0_4th_compact}
	\left(\psi_x\right)_0
	= \left[\dfrac{1}{2\Delta x}-\dfrac{\Delta x}{12}
	\dfrac{2m^*}{\hbar^2}(V_1-E)\right]\psi_1
	-\left[\dfrac{1}{2\Delta x}-\dfrac{\Delta x}{12}
	\dfrac{2m^*}{\hbar^2}(V_0-E)\right]\psi_{-1},
\end{equation}
by setting $i=0$ in
\eqref{Compact_4th_general_first_derivative_third_form},
replacing $f_{\pm 1}$ with $\frac{2m^*}{\hbar^2}
(V_{\pm 1}-E)\psi_{\pm 1}$,
where $V_{-1} = V_0$ according to \eqref{V_extend}.
Similarly, one obtains
\begin{equation}\label{phi_x_Nx_4th_compact}
	\left(\psi_x\right)_{N_x}
	= \left[\dfrac{1}{2\Delta x}-\dfrac{\Delta x}{12}
	\dfrac{2m^*}{\hbar^2}(V_{N_x}-E)\right]\psi_{N_x+1}
	-\left[\dfrac{1}{2\Delta x}-\dfrac{\Delta x}{12}
	\dfrac{2m^*}{\hbar^2}(V_{N_x-1}-E)\right]\psi_{N_x-1}.
\end{equation}
Substituting \eqref{phi_x_0_4th_compact} and
\eqref{phi_x_Nx_4th_compact}
in \eqref{eq:TBC_left} and \eqref{eq:TBC_right},
and combining \eqref{4th_compact_schrodinger_interior}
at boundary grid points $x_0$ and $x_{N_x}$
to eliminate $\psi_{-1}$ and
$\psi_{N_x+1}$, respectively, one obtains
discretizations of
\eqref{eq:TBC_left} and \eqref{eq:TBC_right} respectively
as
\begin{equation}\label{TBC_left_4thCompact}
	a_0\psi_0+b_1\psi_1=d_0,
\end{equation}
and
\begin{equation}\label{TBC_right_4thCompact}
	b_{N_x-1}\psi_{N_x-1}+a_{N_x}\psi_{N_x}=0,
\end{equation}
where

\begin{eqnarray*}
	a_0 &=& -2t+10(E-V_0)+2\ri k_1 t\Delta x\dfrac{t+E-V_0}{t+2(E-V_0)},\\
	b_1 &=& t+E-V_1 + (t+2(E-V_1))\dfrac{t+E-V_0}{t+2(E-V_0)},\\
	d_0 &=& 4\ri k_1 t \Delta x\dfrac{t+E-V_0}{t+2(E-V_0)},\\
	b_{N_x-1} &=& t+E-V_{N_x-1} + (t+2(E-V_{N_x-1}))\dfrac{t+E-V_{N_x}}{t+2(E-V_{N_x})},\\
	a_{N_x} &=& -2t+10(E-V_{N_x})+2\ri k_2 t\Delta x\dfrac{t+E-V_{N_x}}{t+2(E-V_{N_x})},
\end{eqnarray*}
with $t = \frac{\hbar^2}{2m^*}\lambda$.
From now on, we shall call \eqref{TBC_left_4thCompact}
+\eqref{TBC_right_4thCompact}
the compact fourth order transparent boundary conditions(C4TBCs).
Consequently, a globally fourth order compact discretization
scheme is yielded for the Schr{\"o}dinger
problem, i.e., \eqref{4th_compact_schrodinger_interior}
equipped with the C4TBCs.

Although \eqref{4th_compact_schrodinger_interior}
equipped with the C4TBCs
possesses a desirable accuracy order in discretizing
the Schr{\"o}dinger problem,
it introduces spurious oscillation in numerical solution
when $V(x)$ vanishes.
We illustrate this phenomenon through an example.
Let $V(x)\equiv 0$, $E = 0.5$, $L = 10$,
and a wave enter from $x=0$.
Parameters $\hbar=m^*=1$ for simplicity.
Because $V\equiv 0$, the wave is completely transmitted
through $x=L$, thus $\psi^*(x) = \de^{\mathrm{i}kx}$
solves the Schr{\" o}dinger
problem with $k=\sqrt{\frac{2m^* E}{\hbar^2}}$.
Norm of numerical wave functions
of the Schr{\"o}dinger problem
solved with different schemes are collected in \figref{fig01}.
Result corresponding to \eqref{4th_compact_schrodinger_interior}
equipped with the C4TBCs,
is plotted in red dash line.
Corresponding curve of the exact wave function, i.e.,
$\psi^*(x) = \de^{\mathrm{i}kx}$, is also shown
in blue solid line as reference.
Obviously, spurious oscillation occurs
when \eqref{4th_compact_schrodinger_interior}
is equipped with the C4TBCs.

\begin{figure}[htbp]
	\centering{\includegraphics[width=\textwidth]{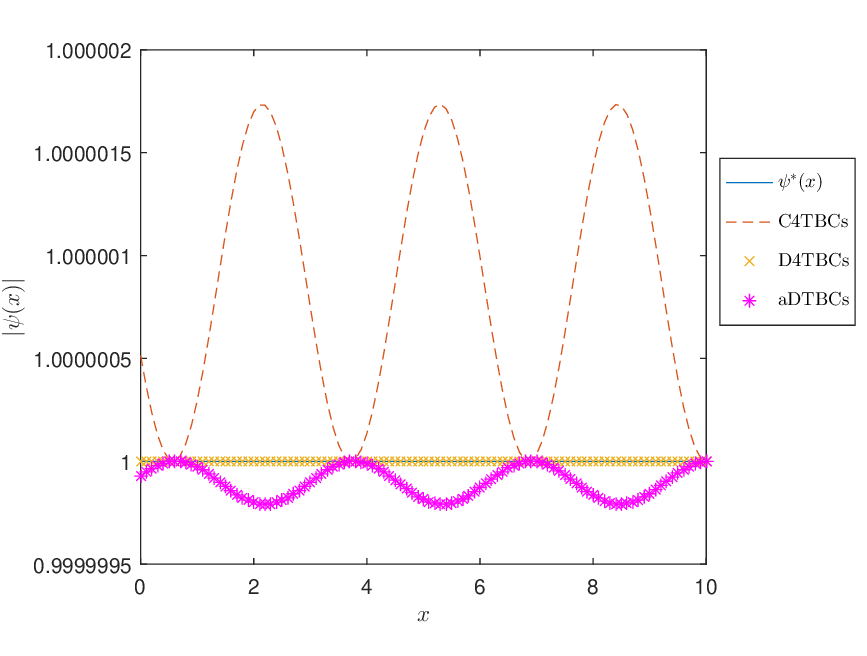}}
	\caption{Norm of numerical wave functions
		of the Schr{\"o}dinger problem, where different
		discrete boundary conditions are compared}\label{fig01}
\end{figure}

Recalling \eqref{wave_outside_boundary}, one notices
it is insufficient to only consider high order
discretizations of the TBCs, however the discrete
form of the real solution is the priority.
Based on the work originally introduced in \cite{Anton2001},
we assume discrete wave solutions for $x\leq 0$ and $x\geq L$:
\begin{equation}\label{discrete_solution_outside_region}
	\psi_j = \left\{\begin{array}{ll}
		\alpha^j,& j \leq 0, \\
		\beta^j,& j \geq N_x.
	\end{array}\right.
\end{equation}
Substituting \eqref{discrete_solution_outside_region}
in \eqref{4th_compact_schrodinger_interior},
and letting $V(x)$ vanishes for $x \notin (0,L)$,
one obtains two distinct values of $\alpha$:

\[
\alpha_\pm= \frac{t -5 (E-V_0) \pm \mathrm{i}
	\sqrt{12(E-V_0)(t - 2(E-V_0))}}{t+(E-V_0)},
\]
and two distinct values of $\beta$:
\[
\beta_\pm = \frac{t -5 (E-V_{N_x}) \pm \mathrm{i}
	\sqrt{12(E-V_{N_x})(t - 2(E-V_{N_x}))}}{t+(E-V_{N_x})}.
\]
Therefore, when $t>\max\left\{2\left(E-V_0\right),
2\left(E-V_{N_x}\right)\right\}$,
$\lv \alpha_\pm\rv = \lv \beta_\pm\rv = 1$ due to
\begin{equation}\label{alpha_beta_wave}\begin{array}{cc}
		\alpha_\pm = \de^{\pm\mathrm{i}\tilde{k}_1 \Delta x}, &
		\beta_\pm = \de^{\pm\mathrm{i}\tilde{k}_2 \Delta x},
\end{array}\end{equation}
with $\tilde{k}_1$ and $\tilde{k}_2$ the approximations
of $k_1$ and $k_2$, respectively.
The so called discrete dispersion relations
are thus obtained as

\begin{equation}\label{dispersion_alpha}
	E -V_0= \frac{t(1-\cos(\tilde{k}_1 \Delta x))}
	{5+\cos(\tilde{k}_1 \Delta x)},
\end{equation}
and
\begin{equation}\label{dispersion_beta}
	E-V_{N_x} = \frac{t(1-\cos(\tilde{k}_2 \Delta x))}
	{5+\cos(\tilde{k}_2 \Delta x)}.
\end{equation}
Taking Taylor's expansions of both right hand sides of
\eqref{dispersion_alpha} and \eqref{dispersion_beta},
one finds out

\begin{equation}\label{E_approximation}
	\begin{array}{cc}
		E - V_0= \frac{\hbar^2 \tilde{k}_1^2}{2m^*}
		+ O(\Delta x^4),&
		E - V_{N_x} = \frac{\hbar^2 \tilde{k}_2^2}{2m^*}
		+ O(\Delta x^4),
	\end{array}
\end{equation}
and subsequently,
\begin{equation}\label{k_approximation}
	\begin{array}{cc}
		\tilde{k}_1 = k_1 + O(\Delta x^4),
		& \tilde{k}_2 = k_2 + O(\Delta x^4).
	\end{array}
\end{equation}
\begin{remark}\label{remark_E_approximation}
	Actually, \eqref{E_approximation} is generalized
	to
	\begin{equation}\label{E_approximation_general}
		\begin{array}{cc}
			E -V_0 = \frac{\hbar^2 \tilde{k}_1^2}{2m^*}
			+ O(\Delta x^p),&
			E - V_{N_x} = \frac{\hbar^2 \tilde{k}_2^2}{2m^*}
			+ O(\Delta x^p),
		\end{array}
	\end{equation}
	when the discretization scheme applied to discretize
	\eqref{eq:schrodinger} has $p$th order accuracy.
	This is because when we take Taylor's expansion of
	the related scheme, and
	substitute $\phi=\de^{\mathrm{i}\tilde{k}_1 x}$, for instance, in the scheme,
	we have
	\[
	\frac{\hbar^2 \tilde{k}_1^2}{2m^*}\phi + V_0 \phi +O(\Delta x^p)\phi = E
	\phi,
	\]
	which indicates the first relation of
	\eqref{E_approximation_general}.
\end{remark}

Substituting \eqref{alpha_beta_wave} in
\eqref{discrete_solution_outside_region},
one gets discrete waves in opposite directions for $x\leq 0$
and $x\geq L$, respectively.
Making linear combinations of the above discrete waves
in opposite directions as described with
\eqref{wave_outside_boundary},
one obtains the approximation of the solution outside
the device:
\begin{equation}\label{discrete_wave_outside}
	\psi_j = \left\{
	\begin{array}{ll}
		R \alpha^{-j} + \alpha^j,
		& j \leq 0, \\
		T \beta^j ,
		&	j \geq N_x,
	\end{array}
	\right.
\end{equation}
where
\begin{equation}\label{def:alpha_beta}
	\begin{array}{cc}
		\alpha \triangleq \alpha_+, &
		\beta \triangleq \beta_+. \end{array}\end{equation}
As a result of \eqref{discrete_wave_outside},
\begin{eqnarray}\label{discrete_BC_left}
	\psi_{-1} - \alpha\psi_0 &=&
	\alpha^{-1} -\alpha, \\\label{discrete_BC_right}
	\psi_{N_x+1} -  \beta \psi_{N_x}&=& 0
\end{eqnarray}
are derived as an alternative discretization scheme
of the TBCs, besides the C4TBCs.

\begin{lemma}\label{lemma:boundary_4_order}
	Let $t > \max\left\{2\left(E-V_0\right),
	2\left( E-V_{N_x}\right)\right\}$ be fulfilled.
	\eqref{discrete_BC_left} and
	\eqref{discrete_BC_right}
	correspond fourth order discretizations
	of \eqref{eq:TBC_left} and \eqref{eq:TBC_right},
	respectively.
\end{lemma}
\begin{proof}
	Taking Taylor's expansion of $\psi(-\Delta x;k)$
	at $x=0$, one sees
	\begin{equation}\label{taylor_DTBCs}
		\psi(-\Delta x;k) = \psi(0;k) - \psi^\prime(0;k)\Delta x
		+ \dfrac{\psi^{\prime\prime}(0;k)}{2!}\Delta x^2
		-\dfrac{\psi^{(3)}(0;k)}{3!}\Delta x^3
		+ \dfrac{\psi^{(4)}(0;k)}{4!}\Delta x^4+O(\Delta x^5).
	\end{equation}
	According to \eqref{wave_outside_boundary}, one gets
	\begin{equation}\label{relation_psi_0}
		\psi^{(2n)}(0;k) = (\mathrm{i}k_1)^{2n}\psi(0;k),
		\quad \psi^{(2n+1)}(0;k) = (\mathrm{i}k_1)^{2n+1}(2-\psi(0;k)).
	\end{equation}
	Substituting \eqref{relation_psi_0} in \eqref{taylor_DTBCs},
	one further obtains
	\begin{equation}\label{taylor_second}
		\begin{aligned}
			\psi(-\Delta x;k) & = -\psi^\prime(0;k)\Delta x+\left[1+\dfrac{(\mathrm{i}k_1\Delta x)^2}{2!}+\dfrac{(\mathrm{i}k_1\Delta x)^3}{3!}+\dfrac{(\mathrm{i}k_1\Delta x)^4}{4!}\right]\psi(0;k)\\
			&\quad + 2\ri\dfrac{(k_1\Delta x)^3}{3!}+O(\Delta x^5)\\
			& = -[\psi^\prime(0;k)+\ri k_1\psi(0;k)-2\ri k_1]\Delta x + \de^{\mathrm{i}k_1\Delta x}\psi(0;k) \\
			& \quad + \de^{-\mathrm{i}k_1\Delta x}-\de^{\mathrm{i}k_1\Delta x}+O(\Delta x^5).
		\end{aligned}
	\end{equation}
	According to \eqref{k_approximation},
	\begin{equation}\label{alpha_order}
		\de^{\mathrm{i}k_1\Delta x}=\alpha + O(\Delta x^5)
	\end{equation}
	is natural, and
	\begin{equation*}
		\psi^\prime(0;k)+\ri k_1\psi(0;k)-2{\mathrm{i}}k_1 =
		-\dfrac{\psi_{-1}-\alpha\psi_0-\alpha^{-1}+\alpha}{\Delta x}+O(\Delta x^4)
	\end{equation*}
	is thus obtained by substituting \eqref{alpha_order}
	in \eqref{taylor_second}.
	Similarly,
	\begin{equation*}
		\psi^\prime(L;k) - \ri k_2\psi(L;k)
		=\dfrac{\psi_{N_x+1}-\beta\psi_{N_x}}{\Delta x}
		+O(\Delta x^4)
	\end{equation*}
	could be deduced similarly.
\end{proof}

\eqref{discrete_BC_left}-\eqref{discrete_BC_right} is the so-called D4TBCs.
According to
{\textbf{Lemma~\ref{lemma:boundary_4_order}},
	one arrives the first main result of this paper, i.e.,
	\textbf{Theorem~\ref{theorem:4_order_globally}}.

	\begin{theorem}\label{theorem:4_order_globally}
		\eqref{4th_compact_schrodinger_interior}
		equipped with the {\rm{D4TBCs}} forms a globally fourth
		order discretization of the {\rm 1D} Schr{\"o}dinger
		problem. \qed
	\end{theorem}

	At the end of this section, we point out
	the set of D4TBCs is a "good" candidate
	for discretizing the TBCs,
	not only because the accuracy order of the D4TBCs
	matches that of \eqref{4th_compact_schrodinger_interior},
	more importantly, but also because the D4TBCs
	essentially avoid spurious oscillation in numerical
	solution when the potential vanishes.
	As an evidence and for comparison, we show
	norm of numerical wave function which solves
	\eqref{4th_compact_schrodinger_interior}
	equipped with D4TBCs
	in {\textbf{Fig.~\ref{fig01}}},
	and one thing is obvious that oscillation
	in the numerical solution brought with the C4TBCs
	vanishes essentially when D4TBCs is used.
	This phenomenon is not accidental,
	since $\alpha^j = \de^{\mathrm{i} \tilde{k} (j\Delta x)}$
	($\tilde{k}=\tilde{k}_1=\tilde{k}_2$)
	is the unique solution of
	\eqref{4th_compact_schrodinger_interior}
	equipped with the D4TBCs when $V(x)\equiv 0$.
	Related conclusion and necessary proof are
	stated in
	{\textbf{Theorem}~\ref{theorem_4th_unique}}.
	
	As a preparation, which is also needed to prove
	{\textbf{Theorem}~\ref{theorem_4th_unique}},
	we introduce a discrete version of integral by parts law,
	i.e., {\textbf{Lemma}~\ref{lemma:discrete_integral}}.
	\begin{lemma}\label{lemma:discrete_integral}
		Assuming $u$ and $v$ are two grid functions defined
		on $\left\{x_i : -1\leq i\leq N_x+1\right\}$,
		one has
		\begin{equation}\label{discrete_integral_by_parts}
			-\Delta x\sum_{i=0}^{N_x}(\delta_x^2 u_i)v_i =
			\Delta x\sum_{i=0}^{N_x+1}\left(\delta_x
			u_{i-\frac{1}{2}}\right)\left(\delta_x
			v_{i-\frac{1}{2}}\right)+(D_{+}u_{-1})v_{-1}-
			(D_{-}u_{N_x+1})v_{N_x+1},
		\end{equation}
		where
		\begin{equation}\label{notations}
			\begin{array}{ll}\vspace{3mm}
				\delta_x v_{i-\frac{1}{2}}=
				\dfrac{1}{\Delta x}(v_i-v_{i-1}), & \delta_x^2
				v_i=\dfrac{1}{\Delta x}(\delta_x v_{i+\frac{1}{2}}-\delta_x v_{i-\frac{1}{2}}),\\
				D_{+}v_i = \dfrac{1}{\Delta x}(v_{i+1}-v_i), & D_{-}v_i = \dfrac{1}{\Delta x}(v_i-v_{i-1}).
			\end{array}
		\end{equation}
	\end{lemma}
	\label{app:proof_discrete_integral_by_parts}
	\begin{proof}
		Clearly,
		\begin{equation*}
			\begin{aligned}
				-\Delta x \sum_{i=0}^{N_x}\left(\delta_x^2 u_i\right) v_i
				&=  -\sum_{i=0}^{N_x}\left(\delta_x u_{i+\frac{1}{2}}
				-\delta_x u_{i-\frac{1}{2}}\right) v_i \\
				&= \sum_{i=0}^{N_x}\left(\delta_x u_{i-\frac{1}{2}}\right) v_i
				-\sum_{i=1}^{N_x+1}\left(\delta_x u_{i-\frac{1}{2}}\right)
				v_{i-1}\\
				&=\sum_{i=0}^{N_x+1}\left(\delta_x u_{i-\frac{1}{2}}\right)
				\left(v_i-v_{i-1}\right)+\left(\delta_x u_{-\frac{1}{2}}\right)
				v_{-1}-\left(\delta_x u_{N_x+\frac{1}{2}}\right) v_{N_x+1} \\
				&= \Delta x \sum_{i=0}^{N_x+1}\left(\delta_x
				u_{i-\frac{1}{2}}\right)\left(\delta_x
				v_{i-\frac{1}{2}}\right)+\left(D_{+} u_{-1}\right) v_{-1}-\left(D_{-}
				u_{N_x+1}\right) v_{N_x+1}.
			\end{aligned}
		\end{equation*}
\end{proof}}
\begin{theorem}\label{theorem_4th_unique}
	Let $t > \max\left\{2\left(E-V_0\right),
	2\left( E-V_{N_x}\right)\right\}$ be fulfilled.
	\eqref{4th_compact_schrodinger_interior} equipped with
	the {\rm{D4TBCs}} has unique solution for arbitrary $V(x)$.
\end{theorem}
\begin{proof}
	Using \eqref{notations}, \eqref{4th_compact_schrodinger_interior}
	and the {\rm{D4TBCs}} are rewritten into
	\begin{eqnarray*}
		\delta_x^2 \psi_j -\dfrac{\Delta x^2}{12}\delta_x^2 f_j
		&=& f_j,
		\quad  j = 0,1,\cdots, N_x,\\
		\psi_{-1}-\alpha\psi_0 &=& \alpha^{-1}-\alpha,\\
		\psi_{N_x+1}-\beta \psi_{N_x}&=&0.
	\end{eqnarray*}
	Obviously, uniqueness of the solution of
	\eqref{4th_compact_schrodinger_interior}
	equipped with the {\rm{D4TBCs}} is equivalent to
	\begin{eqnarray}\label{compact_second}
		\delta_x^2 \psi_j -\dfrac{\Delta x^2}{12}\delta_x^2 f_j &=& f_j,
		\quad  j=0,1,\cdots,N_x,\\ \label{homogeneous_BC_left}
		\psi_{-1}-\alpha\psi_0 &=& 0,\\ \label{homogeneous_BC_right}
		\psi_{N_x+1}-\beta \psi_{N_x}&=&0
	\end{eqnarray}
	only has a trivial solution.
	
	Multiplying \eqref{compact_second}
	by $-\Delta x \psi_j^*$, where $\psi_j^*$ denotes
	by the conjugate of $\psi_j$,
	and summing all equations with respect to $j$ for $j\in
	\{0,1,\cdots,N_x\}$, we obtain
	\begin{equation}\label{theorem2_equ00}
		-\Delta x\sum_{j=0}^{N_x}(\delta_x^2\psi_j)\psi_j^* + \dfrac{\Delta x^2}{12}\Delta x\sum_{j=0}^{N_x}(\delta_x^2f_j)\psi^*_j =-\Delta x\sum_{j=0}^{N_x}f_j\psi_j^*.
	\end{equation}
	Applying \eqref{discrete_integral_by_parts} to the two summations
	on the left of \eqref{theorem2_equ00}, one reorganizes the
	result as
	\begin{equation}\label{theorem2_equ01}
		A_1 = B_1 + C_1,
	\end{equation}
	where
	\begin{eqnarray*}
		A_1 &=& \Delta x\sum_{j= 0}^{N_x+1}\left(\delta_x\psi_{j-\frac{1}{2}}\right)\left(\delta_x\psi^*_{j-\frac{1}{2}}\right)+\Delta x\sum_{j=0}^{N_x}f_j\psi_j^*,\\
		B_1 &=& \dfrac{\Delta x^2}{12}\Delta x\sum_{j=0}^{N_x+1}\left(\delta_x f_{j-\frac{1}{2}}\right)\left(\delta_x\psi^*_{j-\frac{1}{2}}\right),\\
		C_1 &=& \dfrac{\Delta x^2}{12}\left[(D_{+}f_{-1})\psi^*_{-1}-(D_{-}f_{N_x+1})\psi^*_{N_x+1} \right]-(D_{+}\psi_{-1})\psi^*_{-1}+(D_{-}\psi_{N_x+1})\psi^*_{N_x+1}.
	\end{eqnarray*}
	At the same time, we multiply both sides
	of \eqref{compact_second}
	with $ \frac{\Delta x^2}{12}(-\Delta x)f_j^*$,
	sum the yielded equations again with respect to $j$
	for $j\in\{0,1,\cdots,N_x\}$, and obtain
	\begin{equation}\label{theorem2_equ02}
		A_2  = B_2 + C_2,
	\end{equation}
	where
	\begin{eqnarray*}
		A_2 &=& \dfrac{\Delta x^4}{144} \Delta x\sum_{j= 0}^{N_x+1}\left(\delta_x f_{j-\frac{1}{2}}\right)\left(\delta_x f^*_{j-\frac{1}{2}}\right)-\dfrac{\Delta x^2}{12}\Delta x\sum_{j=0}^{N_x}f_j f_j^*,\\
		B_2 &=& \dfrac{\Delta x^2}{12}\Delta x\sum_{j=0}^{N_x+1}\left(\delta_x \psi_{j-\frac{1}{2}}\right)\left(\delta_x f^*_{j-\frac{1}{2}}\right),\\
		C_2 &=& \dfrac{\Delta x^2}{12}\left[(D_{+}\psi_{-1}) f^*_{-1}-(D_{-}\psi_{N_x+1}) f^*_{N_x+1} \right]-\dfrac{\Delta x^4}{144}\left[(D_{+}f_{-1})f^*_{-1}-(D_{-}f_{N_x+1})f^*_{N_x+1}\right].
	\end{eqnarray*}
	Obviously, $A_1,A_2\in\mathbb{R}$ and $B_1 = B_2^*$,
	which indicates
	\begin{equation}\label{eq:image}
		\mathrm{Im}(C_1+C_2) = 0.
	\end{equation}
	By using \eqref{homogeneous_BC_left}-\eqref{homogeneous_BC_right}, one concludes that \eqref{eq:image} is equivalent to
	\begin{equation}\label{eq:image2}
		\mathrm{Im}\left(g_0|\psi_{-1}|^2\alpha^{-1}
		+g_{N_x}|\psi_{N_x+1}|^2\beta^{-1}\right)=0,
	\end{equation}
	where
	\begin{equation*}
		\begin{array}{cc}
			g_0 = \left(\frac{V_0-E}{t}-1\right)^2 > 0 ,&
			g_{N_x} =  \left(\frac{V_{N_x}-E}{t}-1\right)^2 > 0,
		\end{array}
	\end{equation*}
	when $t >
	\max\left\{2 \left( E-V_0\right), 2\left( E-V_{N_x}\right)\right\}$.
	Recalling \eqref{def:alpha_beta}, one notices
	the imaginary parts
	of $\alpha$ and $\beta$ have the same signs, thus
	\begin{equation*}
		g_0|\psi_{-1}|^2 = g_{N_x}|\psi_{N_x+1}|^2=0,
	\end{equation*}
	which implies
	\begin{equation}\label{psi_equal_zero}
		\psi_{-1}=\psi_{N_x+1}=0.
	\end{equation}
	Substituting \eqref{psi_equal_zero} in
	\eqref{compact_second}-\eqref{homogeneous_BC_right}, one
	finds out \eqref{compact_second}-\eqref{homogeneous_BC_right}
	admits a unique solution, which concludes the proof of the theorem.
\end{proof}

\subsection{The aDTBCs}
Based on the discussions made in  Sect.~\ref{subsec:essentially_non_oscillating},
pros and cons of the D4TBCs are quite clear.
In one aspect, the accuracy order of the D4TBCs
is consistent with the order of the interior scheme,
which is delightful.
More importantly, spurious oscillation
in numerical solution is avoided, when
$V(x) \equiv 0$.
However, in the other aspect, an extra discrete dispersion
relation is needed via solving an algebraic
equation, with the order of the algebraic equation
increased as the accuracy order of discretization scheme
for the Schr{\"o}dinger equation increases.
At the same time, the dispersion
relation brought numerical error, so does
the D4TBCs consequently.

Actually, TBCs could be discretized without introducing
discretization error, as long as $\psi(x;k)$ is analytic near
physical boundaries.
In the rest of this section, we always assume
$\psi(x;k)$ be analytic near both physical boundaries.
Taking \eqref{eq:TBC_left} for instance,
we use Taylor series with infinite terms
to accurately express $\psi(-\Delta x;k)$:
\begin{equation*}
	\begin{aligned}
		\psi(-\Delta x;k) &= \psi(0;k) - \psi^\prime(0;k) \Delta x
		+ \dfrac{\psi^{\prime\prime}(0;k)}{2!}\Delta x^2
		-\dfrac{\psi^{(3)}(0;k)}{3!}\Delta x^3+\dfrac{\psi^{(4)}(0;k)}{4!}\Delta x^4\\
		& \quad -\dfrac{\psi^{(5)}(0;k)}{5!}\Delta x^5+\cdots\\
		& = \left[1+\dfrac{(\ri k_1\Delta x)^2}{2!}+\dfrac{(\ri k_1\Delta x)^3}{3!}
		+\dfrac{(\ri k_1\Delta x)^4}{4!}+\dfrac{(\ri k_1\Delta x)^5}{5!}+\cdots\right]\psi(0;k)\\
		& \quad -\Delta x\psi^\prime(0;k) -2\ri\left[-\dfrac{(k_1\Delta x)^3}{3!}
		+\dfrac{(k_1\Delta x)^5}{5!}-\cdots\right]\\
		& = -\Delta x\left[\psi^\prime(0;k)+\ri k_1\psi(0;k)-2\ri k_1\right]
		+\de^{\ri k_1\Delta x}\psi(0;k)-2\ri \sin(k_1\Delta x),
	\end{aligned}
\end{equation*}
where \eqref{relation_psi_0} is noticed.
As a result, a discretization scheme for
\eqref{eq:TBC_left} without discretization error is yielded as
\begin{equation}\label{eq:DTBC_real_left}
	\psi_{-1} - \de^{\mathrm{i} k_1 \Delta x} \psi_0
	+ 2\mathrm{i} \sin\left(k_1 \Delta x\right) = 0,
\end{equation}
where $k_1$ and $E$ satisfy the exact dispersion
relation, i.e., $k_1 = \sqrt{\frac{2m^*(E-V_0)}{\hbar}}$.
Similarly, the analytic discretization of \eqref{eq:TBC_right}
reads
\begin{equation}\label{eq:DTBC_real_right}
	\psi_{N_x+1}-\de^{\mathrm{i}k_2\Delta x}\psi_{N_x}=0,
\end{equation}
which does not introduce any discretization error, too.
\eqref{eq:DTBC_real_left}-\eqref{eq:DTBC_real_right} is a specific example of the aDTBCs. Clearly, the overall
accuracy order of \eqref{4th_compact_schrodinger_interior}
equipped with the aDTBCs totally depends on
the appearance of the discretization scheme applied
to discretize the Schr{\"o}dinger equation.
The result of the uniqueness of the
solution of \eqref{4th_compact_schrodinger_interior}
equipped with the {\rm{aDTBCs}} is provided
in {\textbf{Theorem}~\ref{theorem_aDTBCs_unique}},
where the proof is omitted due to its similarity
with the proof of {\textbf{Theorem}~\ref{theorem_4th_unique}}.

\begin{theorem}\label{theorem_aDTBCs_unique}
	Let $t > \max\left\{2 \left( E-V_0\right), 2\left( E-V_{N_x}\right)\right\}$ be
	fulfilled.
	\eqref{4th_compact_schrodinger_interior} equipped with
	the {\rm{aDTBCs}} has unique solution for
	arbitrary $V(x)$. \qed
\end{theorem}

We make comparisons between the two discretization schemes:
\eqref{4th_compact_schrodinger_interior} equipped with
the D4TBCs and the aDTBCs, respectively, by comparing
their accuracy orders via solving the example studied in
 Sect.~\ref{subsec:essentially_non_oscillating}.
Numerical results given in {\textbf{Table~\ref{table1}} show
	that fourth order accuracy is globally achieved by
	\eqref{4th_compact_schrodinger_interior}
	equipped with either one of the D4TBCs and the aDTBCs.

	\begin{table}
	\caption{Convergence orders of	\eqref{4th_compact_schrodinger_interior}equipped with the D4TBCs and the aDTBCs, respectively}
	\label{table1}       
	\centering
	\begin{tabular}{llll}
	\hline\noalign{\smallskip}
	BCs   & $N_x = 200$  & $N_x= 400$  &$N_x = 800$  \\
	\noalign{\smallskip}\hline\noalign{\smallskip}
	D4TBCs    &4.00043 & 3.99986& 4.00085  \\
	aDTBCs   &4.00187  & 4.00137  &4.00107 \\
	\noalign{\smallskip}\hline
	\end{tabular}
	\end{table}
	
	As depicted in {\textbf{Fig.}~\ref{fig01}},
	spurious oscillation exists in the numerical solution
	when $V\equiv 0$, too, when aDTBCs is applied.
	However, apparent reduction in amplitude
	of the spurious oscillation occurs, when the
	aDTBCs is compared to the C4TBCs.
	More importantly, it is not clear whether spurious oscillation
	still exists when $V(x)\nequiv 0$, by using any of boundary
	discretization schemes discussed in this paper, including
	the D4TBCs.
	
	Thanks to the aDTBCs, one is able to improve the
	overall accuracy of the discretization
	scheme for the Schr{\"o}dinger problem to arbitrarily high order.
	Applying \eqref{Compact_Scheme_general_second_derivative}
	to discretize \eqref{eq:schrodinger} at grid points
	$x_i$, $i=0,1,\cdots,N_x$, one finds out it is
	necessary to express $\left(\psi_{xx}\right)_i$ and
	$\psi_i$ for $i<0$ and $i>N_x$ with interior or
	with boundary nodal values of $\psi$
	to close the equations.
	Noticing $\psi$ solves \eqref{eq:schrodinger}
	outside the device with \eqref{V_extend} satisfied,
	i.e., $\left(\psi_{xx}\right)_{i<0(i>N_x)} =
	\frac{2m^*}{\hbar^2}(V_{i<0(i>N_x)}-E)\psi$,
	where $V_{i<0}=V_0$ and $V_{i>N_x}=V_{N_x}$,
	one observes unknowns outside the device region
	really needed to be
	expressed are $\{\psi_j,j=-1,\cdots,-L\}$ and
	$\{\psi_{N_x+j},j=1,\cdots,L\}$, with $L= \max\{l,s\}$.
	Similar to the way that adopted in deriving
	\eqref{eq:DTBC_real_left},
	we consider Taylor series of $\psi(-j \Delta x;k)$,
	$j=1,2,\cdots,L$, with infinite terms,
	and deduce two sets of accurate
	discretizations of the TBCs, which read
	\begin{equation}\label{exterior_left}
		\psi_{-j} - \de^{\mathrm{i}k_1(j\Delta x)}\psi_0 + 2 \mathrm{i}
		\sin\left(k_1 (j\Delta x)\right) = 0, \quad j=1,2,\cdots,L,
	\end{equation}
	and
	\begin{equation}\label{exterior_right}
		\psi_{N_x+j} - \de^{\mathrm{i}k_2(j\Delta x)}\psi_{N_x}  = 0,
		\quad j=1,2,\cdots,L.
	\end{equation}
	\eqref{exterior_left} and \eqref{exterior_right} stand for a 
	general framework of the aDTBCs, with which the overall accuracy  
	of discretization for the 1D Schr{\" o}dinger problem
	could be improved to arbitrarily high levels. 
	
	However, DTBCs consistent with arbitrarily high order
	interior discretization
	are essentially impossible to be deduced for $L\geq 3$.
	The reason for the above assertion is that to derive the
	required DTBCs, one needs to analytically solve a
	high order(at least $6$th order)
	algebraic equation, in order to obtain the
	corresponding discrete dispersion
	relations, which is theoretically impossible.
	Therefore, the aDTBCs could be used as an alternative
	choice besides DTBCs, when $V\nequiv 0$ in
	simulating realistic devices or high order schemes
	are desired.

	\section{The iterative scheme for solving the
		Schr{\"o}dinger-Poisson problem}\label{sec:Algorithm}
	Firstly, we instantiate \eqref{model_problem} with
	the Poisson problem.
	Applying \eqref{model_equation_compact} to discretize
	\eqref{eq:poisson} at interior grid points, one obtains
	
	\begin{equation}\label{Poisson_discretization}
		\lambda (V_{s})_{i-1}-f_{i-1}-\left[2\lambda
		(V_{s})_{i}+10f_i\right]+\lambda
		(V_{s})_{i+1}-f_{i+1}=0, \cdots, i=1,\cdots,N_x-1,
	\end{equation}
	where $(V_{s})_i$ denotes the approximation
	of $V_{s}(x_i)$, and
	\begin{equation}\label{def:f_i}
		f_i = \frac{q_e^2}{\eps}(N_d(x_i)-n_i),
	\end{equation}
	where $n_i$ is the approximation of $n(x)$ at $x_i$.
	According to \eqref{total_density}, the density of electron
	is defined as an integral on the real line,
	thus truncation of the integral is necessary
	in numerical calculations.
	For $E_F$ used in all of the numerical experiments,
	one finds out $F(E_F-E) $, which is given by \eqref{fermi}, is less than $10^{-10}$ when $E>0.8(\rm eV)$.
	Therefore, we truncate
	the integral range into $[0,0.8](\rm eV)$ and
	calculate $n_i$ via the
	adaptive Simpson quadrature routine\cite{1969Notes}.
	
	One further applies
	\eqref{Compact_4th_general_first_derivative_third_form} to
	discretize \eqref{poisson_bc} at $x_0$ and $x_{N_x}$, respectively,
	and obtains the corresponding discretization
	schemes in the form of
	\begin{equation}\label{poisson_dbc_left}
		\frac{(V_{s})_1-(V_{s})_{-1}}{2\Delta x}
		- \frac{\Delta x}{12} \left(f_1-f_{-1}\right)=0,
	\end{equation}
	and
	\begin{equation}\label{poisson_dbc_right}
		\frac{(V_{s})_{N_x+1}-(V_{s})_{N_x-1}}{2\Delta x}
		- \frac{\Delta x}{12} \left(f_{N_x+1}-f_{N_x-1}\right)=0.
	\end{equation}
	Applying
	\eqref{Poisson_discretization} to discretize
	\eqref{eq:poisson} at $x_0$ and $x_{N_x}$,
	and combining the results with
	\eqref{poisson_dbc_left} and
	\eqref{poisson_dbc_right} respectively, one closes
	\eqref{Poisson_discretization} with fourth
	order discretization of \eqref{poisson_bc}, i.e.,
	\begin{equation}\label{poisson_dbc}
		\begin{array}{ccc}
			-2\lambda (V_{s})_0 + 2\lambda (V_{s})_1
			&=& -f_{-1}+10f_0+3f_1 ,\\
			2\lambda (V_{s})_{N_x-1}-2\lambda
			(V_{s})_{N_x} &=& 3f_{N_x-1}+10f_{N_x}-f_{N_x+1}.
		\end{array}
	\end{equation}
	where $f_{-1}$ and $f_{N_x+1}$ are defined according to
	\eqref{def:f_i}, and approximated with the adaptive
	Simpson quadrature mentioned previously,
	where $\psi_{-1}$ and $\psi_{N_x+1}$ involved in the
	integrals are determined through the D4TBCs or
	the aDTBCs.
	As a result, a globally fourth order discretization of
	the 1D Schr{\"o}dinger-Poisson problem
	is proposed by equipping
	\eqref{4th_compact_schrodinger_interior}
	with the D4TBCs or the aDTBCs, and coupling
	\eqref{Poisson_discretization} equipped with
	\eqref{poisson_dbc}.
	For the convenience of subsequent discussion,
	we name \eqref{4th_compact_schrodinger_interior}
	with the D4TBCs and the aDTBCs as the DS1 and
	the DS2, respectively, and the whole 1D
	discrete Schr{\"o}dinger-Poisson
	system involving the D4TBCs and the aDTBCs
	as the DSP1 and the DSP2, respectively.

	The details of the simulation process are given in \textbf{Algorithm~\ref{iteration_scheme}}}.

\begin{algorithm}\label{iteration_scheme}
	
	\item[1] Make an initial guess about $V_{s}(x)$.
	For example, the setup is $V_{s}(x)\equiv 0$
	for RTD simulation.
	\item[2] For any given energy $E$, solve
	\eqref{4th_compact_schrodinger_interior}
	equipped with either the {\rm{D4TBCs}} or the {\rm{aDTBCs}}
	to obtain corresponding discrete wave functions.
	\item[3] If $V(x)$ is known a {\textit{priori}},
	then terminate the iteration.
	Otherwise, calculate $n_i$, $i=-1,0,\cdots,N_x+1$,
	via the adaptive Simpson quadrature routine, where discrete wave functions
	corresponding to plenty of sampling values
	of $E$ should be solved repeatedly according to step 2.
	Then apply Newton-Raphson iteration \cite{2003nanoMOS,1997Single} to solve
	\eqref{Poisson_discretization} equipped with
	\eqref{poisson_dbc}.
	The criterion to end this Newton-Raphson iteration is
	\begin{equation}\label{criterion}\| (\bV_{s})^{(p+1)}
		- (\bV_{s})^{(p)}\|_{L^\infty} \leq 10^{-10},\end{equation}
	where $p$ denotes the iteration step.
	\item[4] Repeat step 2 and step 3 till the numerical solution of the
	Schr{\"o}dinger-Poisson problem converges, where
	the criterion for the convergence is the same as
	that used in the 3rd step, where $p$ is the
	iteration step of the whole iteration.
\end{algorithm}

\section{Numerical experiments}\label{sec:numerical_experiments}
In this section, we use several examples, including
a short $n^{++}-n^+-n^{++}$
resistor and two RTDs, to validate
theoretical accuracy orders of the discretizations,
and the ability of {\bf Algorithm
	\ref{iteration_scheme}} in simulating quantum devices.

\subsection{A one dimensional
	$n^{++}-n^+-n^{++}$ resistor}\label{section:toymodel}
\begin{figure}[htbp]
	\centering
	\includegraphics[width=0.8\textwidth]{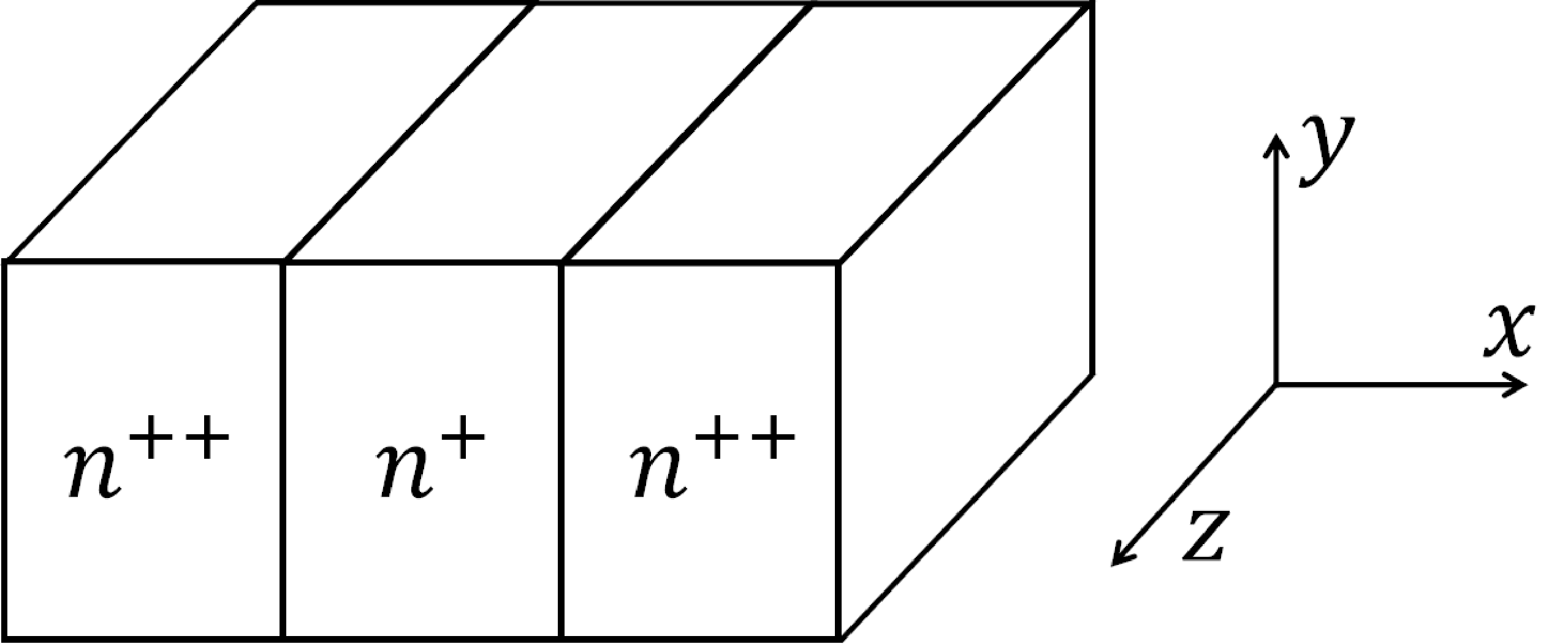}
	\caption{Skeleton of a three dimensional
		$n^{++}-n^+-n^{++}$ resistor, where
		the transport of electrons in the x-direction
		is of concern}\label{fig:device}
\end{figure}

In \figref{fig:device}, the skeleton of a short
three dimensional $n^{++}-n^+-n^{++}$ resistor is shown.
We further assume all physical quantities involved
are homogeneous in both the $y$-direction and the
$z$-direction, thus the problem could be simplified
as simulation of a one dimensional resistor, where
the appearance of electrons in the $x$-direction
is of concern.
Parameters are set as that used in \cite{2000Nanoscale},
which read: $m = 0.25 m_0$,
where $m_0 = 9.1\times 10^{-31}\ (\rm Kg)$,
$\varepsilon = 10\varepsilon_0$, where
$\varepsilon_0 = 8.85\times 10^{-12}\
(\rm F\ m^{-1})$, $T_L = 300\ (\rm K)$,
and $E_F = 0.318\ (\rm eV)$. The device is
$30\ (\rm nm)$ long in total and  discretized
with a uniform grid with $100$ grid points.
$V_b\equiv 0$ and the doping profile reads
\begin{equation}
	N_d(x) = \left\{
	\begin{array}{ll}
		10^{20}\ \rm cm^{-3},\quad & x \in [0,4.5]\cup [25.5,30](\rm nm),\\
		5\times 10 ^{19}\ \rm cm^{-3},& x\in (4.5,25.5)(\rm nm).\\
	\end{array}
	\right.
\end{equation}

The simulation is carried out following {\bf Algorithm
	\ref{iteration_scheme}}, where the numerical results are
in \figref{fig:toymodel},
\figref{fig:error}, \figref{fig:toymodel_IV},
and {\textbf{Table~\ref{table:convergence}}.
	Appearance of {\bf Algorithm
		\ref{iteration_scheme}} in simulating the
	1D $n^{++}-n^+-n^{++}$ resistor is studied, by taking
	results obtained via solving the problem
	yielded by applying the second order NEGF method\cite{2000Nanoscale} as benchmarks.
	Numerical results about the potential functions and
	densities of electrons are shown in \figref{fig:toymodel}
	with $V_{ds}=0\ \rm (V)$ and $V_{ds}=0.25\ \rm (V)$,
	respectively.
	Vividly, curves simulated following
	\textbf{ Algorithm~\ref{iteration_scheme}} match
	that solved with the second order NEGF method.

	\begin{figure}[htbp]
		\centering  
		\subfigure[$V_{ds} = 0\ \rm (V)$]{
			\label{Fig.sub.1}
			\includegraphics[width=0.45\textwidth]{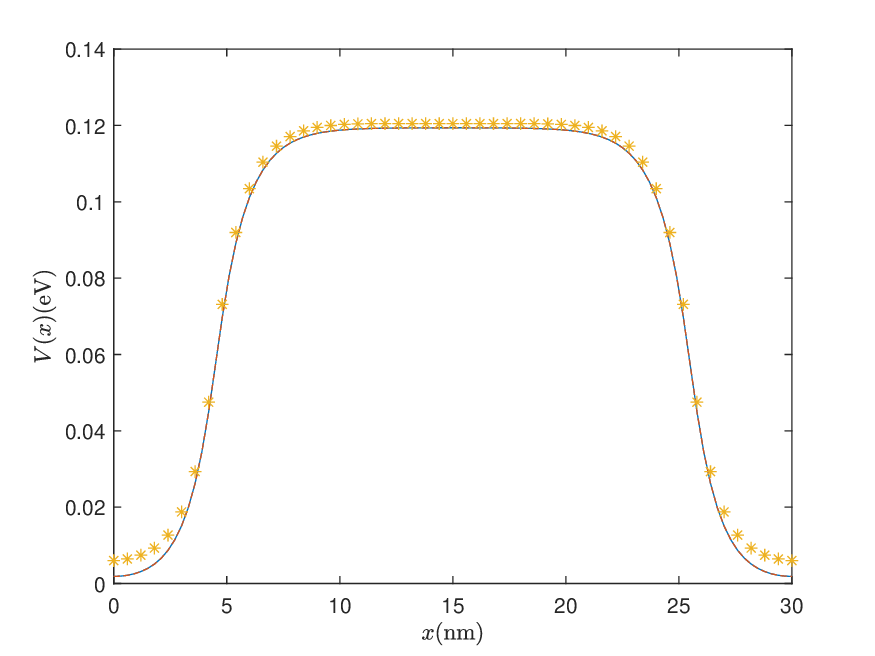}}
		\subfigure[$V_{ds} = 0\ \rm (V)$]{
			\label{Fig.sub.2}
			\includegraphics[width=0.45\textwidth]{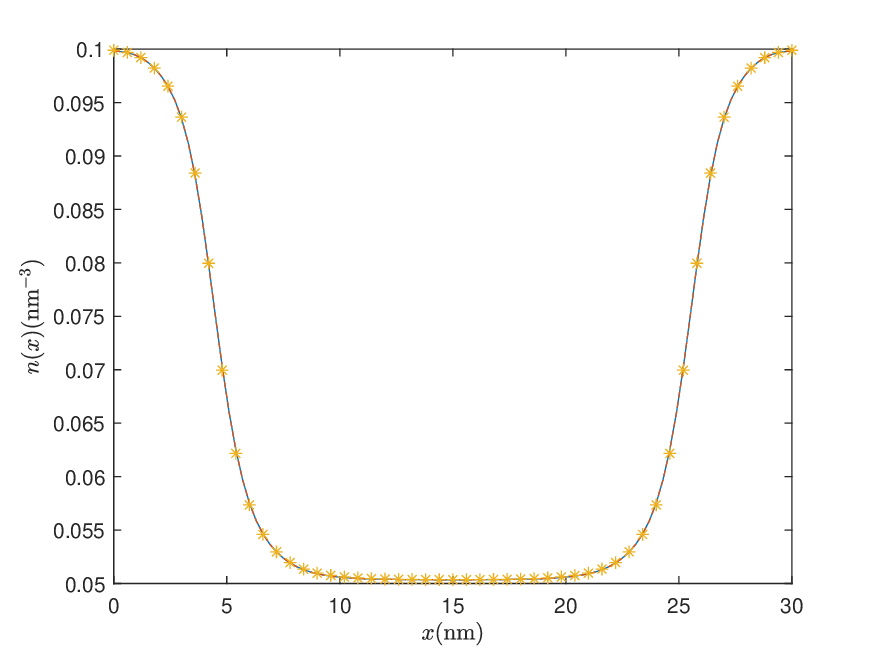}}
		\subfigure[$V_{ds} = 0.25\ \rm (V)$]{
			\label{Fig.sub.3}
			\includegraphics[width=0.45\textwidth]{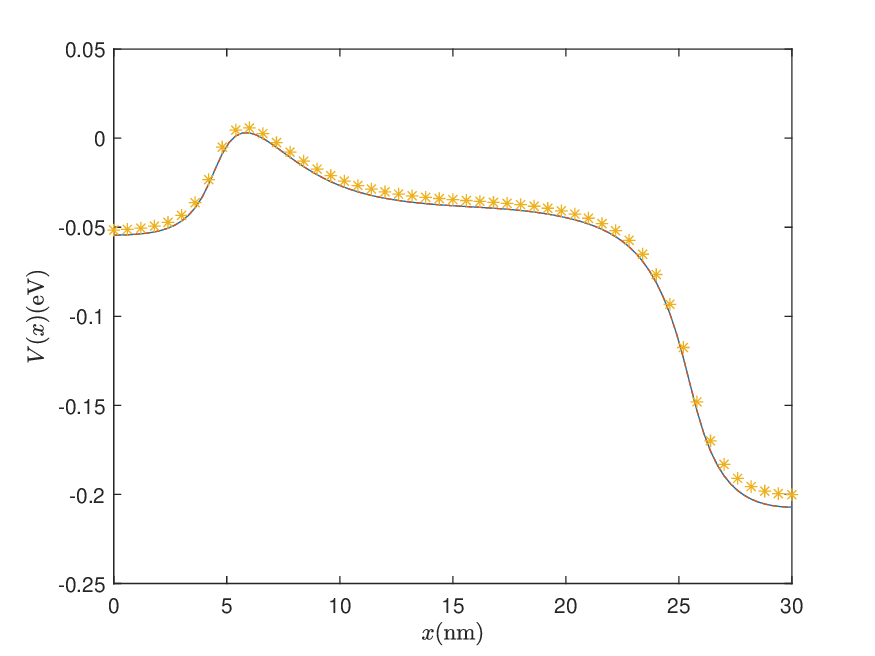}}
		\subfigure[$V_{ds} = 0.25\ \rm (V)$]{
			\label{Fig.sub.4}
			\includegraphics[width=0.45\textwidth]{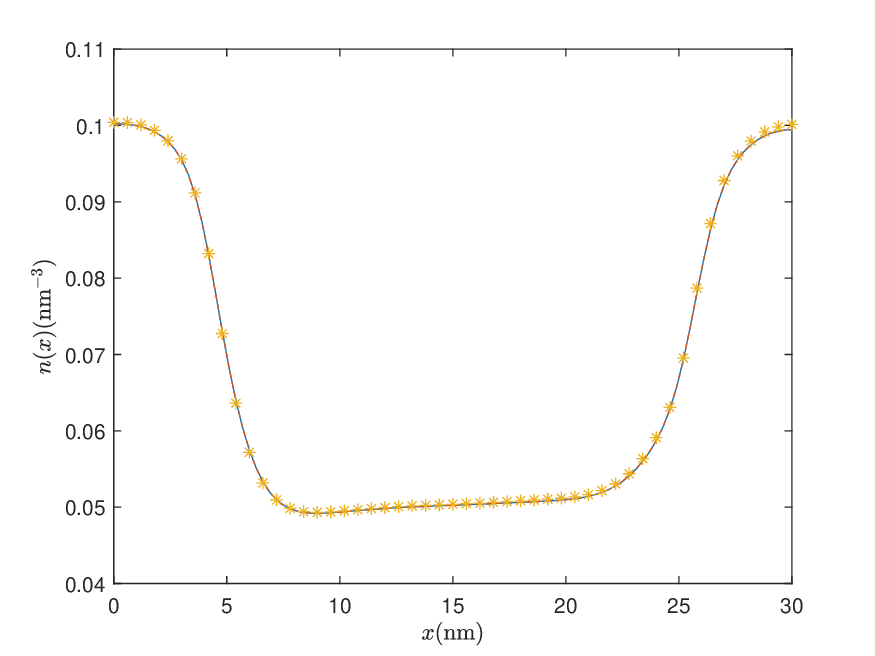}}
		\caption{Numerical results about the
			potential functions $V(x)$(on the left)
			and density
			functions $n(x)$(on the right)
			simulated with two samplings of
			$V_{ds}$, i.e., $0\ \rm (V)$ and $0.25\ \rm (V)$,
			respectively.
			The DSP1(in blue solid line)
			and the DSP2(in red dash line) are both considered,
			for both values of $V_{ds}$, separately.
			Corresponding results(in star) solved with
			the second order NEGF method are shown as benchmarks}
		\label{fig:toymodel}
	\end{figure}

	In order to examine the
	theoretical accuracy orders of the DSP1 and the DSP2,
	respectively, we smooth the doping density function
	with
	\begin{equation}\label{eq:convolution}
		\phi(x) = \dfrac{5\de^{5x}}{(1+\de^{5x})^2},
	\end{equation}
	and $V_{ds} = 0\ \rm (V)$ during the simulation.
	We define error vector $\boldsymbol{\delta} \bV_{\Delta x}$ as the
	absolute error of $\bV_{\Delta x}$ with respect to $\bV^*$,
	where $\bV_{\Delta x}$ denotes the potential function solved on
	a grid with grid space $\Delta x$, and $\bV^*$ is the projection
	of $\bV_{L/1600}$ in the space where $\bV_{\Delta x}$ lies in.
	The error function of the numerical solution is defined as
	\begin{equation}\label{def:error}
		E_{\Delta x} = \max \boldsymbol{\delta} \bV_{\Delta x},
	\end{equation}
	and the accuracy order is estimated by calculating
	$\log_2\left(\frac{E_{\Delta x}}{E_{\Delta x/2}}\right)$
	with gradually halved $\Delta x$.
	Numerical results about the convergence behavior
	and convergence orders of DSP1 and DSP2 are
	separately shown in \figref{fig:error} and
	{\textbf{Table~\ref{table:convergence}}}, which
	both indicate good convergence of them.

	\begin{table}
	\caption{Convergence orders of DSP1 and DSP2}
	\label{table:convergence}       
	\centering
	\begin{tabular}{lllll}
	\hline\noalign{\smallskip}
	Method   & $N_x = 100$  & $N_x= 200$  &$N_x = 400$  & $N_x=800$  \\
	\noalign{\smallskip}\hline\noalign{\smallskip}
	DSP1    &4.0475   & 3.9906 & 4.0073  & 4.0815\\
	DSP2    &  4.0178  & 3.9997  &4.0052   &4.0859 \\
	\noalign{\smallskip}\hline
	\end{tabular}
	\end{table}
	
	\begin{figure}[htbp]
		\centering  
		\subfigure[DSP1]{
			\label{Fig4.sub.a}
			\includegraphics[width=0.45\textwidth]{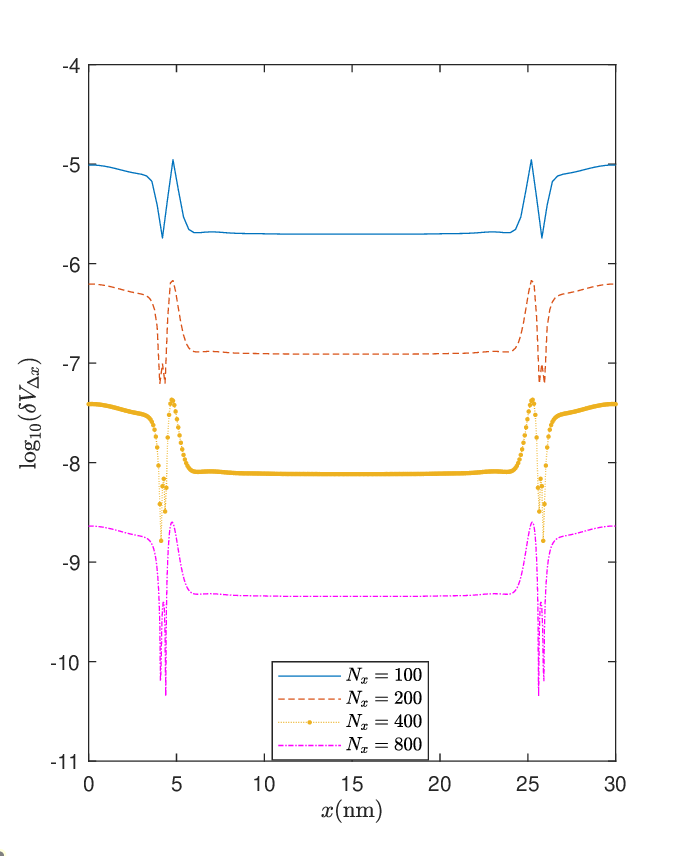}}
		\subfigure[DSP2]{
			\label{Fig4.sub.b}
			\includegraphics[width=0.45\textwidth]{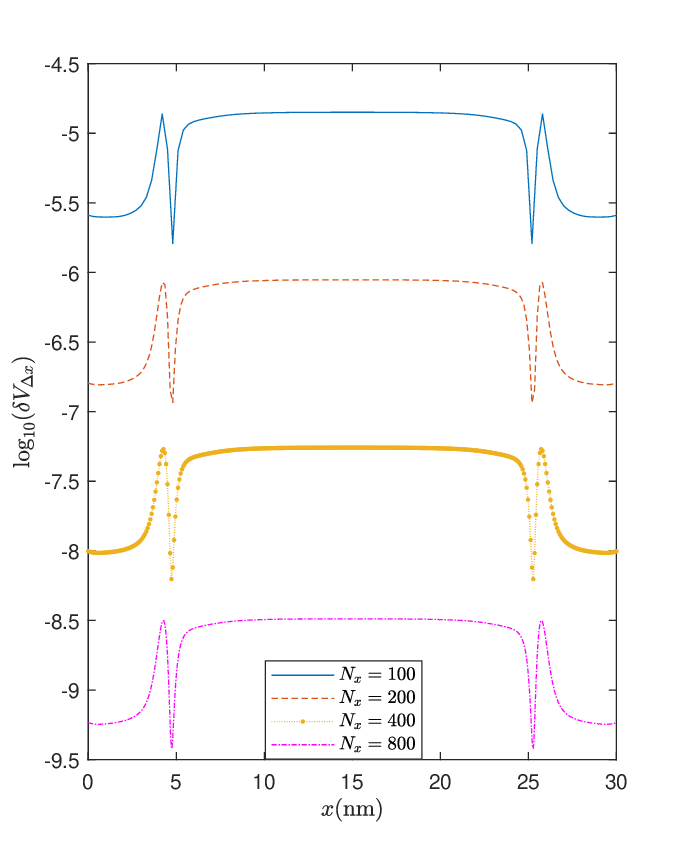}}
		\caption{Tendency of $\log_{10}\left(\boldsymbol{\delta} \bV_{\Delta x}\right)$
			with respect to decreasing $\Delta x$, where
			the DSP1(a) and the DSP2(b) are considered
			separately}
		\label{fig:error}
	\end{figure}
	
	At last, we show the I-V characteristic curves of the
	1D $n^{++}-n^+-n^{++}$ resistor in \figref{fig:toymodel_IV},
	which are simulated with {\textbf{Algorithm~\ref{iteration_scheme}}},
	by considering the DSP1 and the DSP2 in the algorithm,
	respectively.
	Constant resistance values are observed
	in both simulations, which is a characteristic
	of classic devices.
	Furthermore, convergence of I-V characteristic curves
	with respect to decreasing grid sizes indicates
	convergence of both the DSP1 and the DSP2.

	\begin{figure}[htbp]
		\centering  
		\subfigure[DSP1]{
			\label{Fig5.sub.a}
			\includegraphics[width=0.45\textwidth]{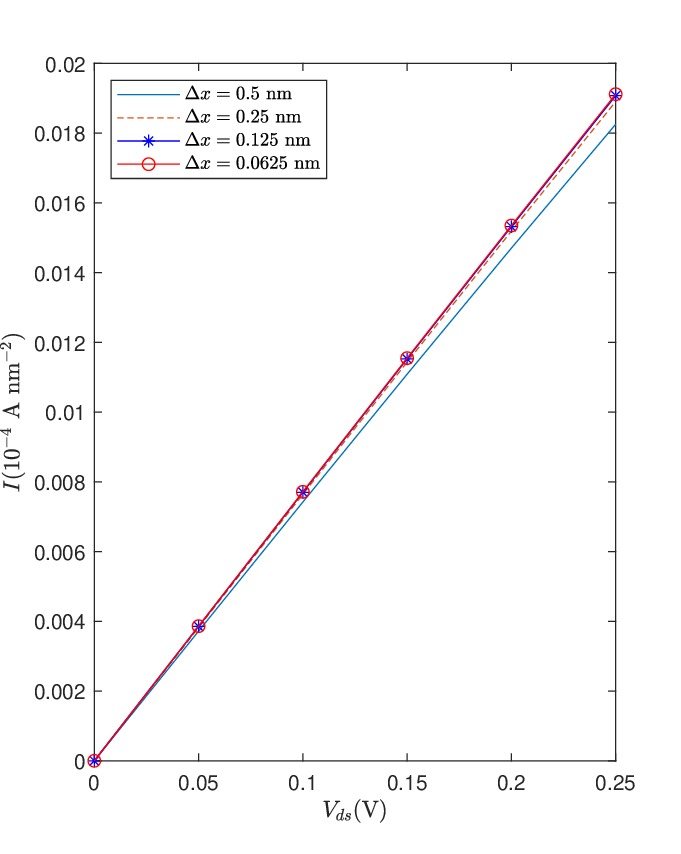}}
		\subfigure[DSP2]{
			\label{Fig5.sub.b}
			\includegraphics[width=0.45\textwidth]{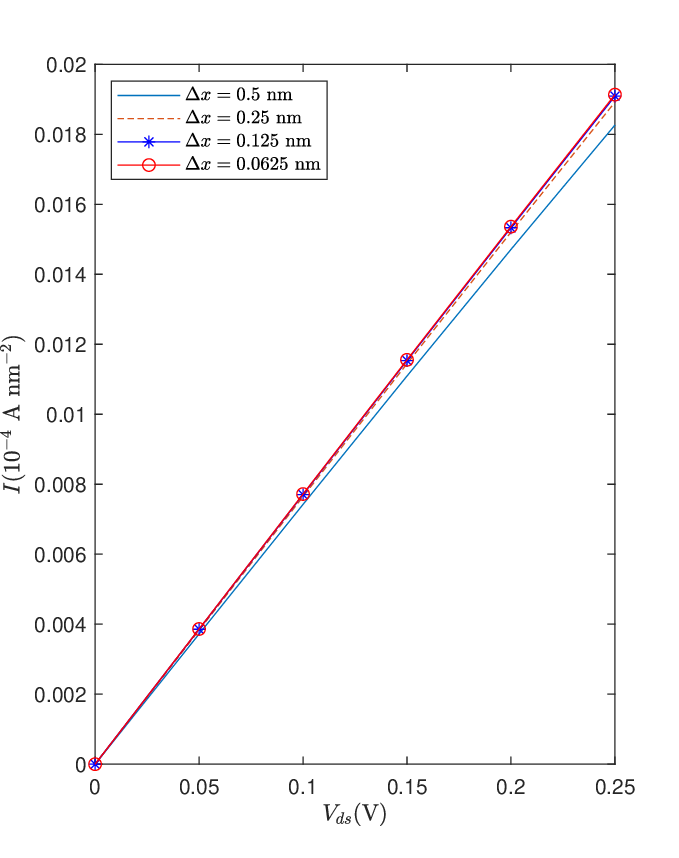}}
		\caption{I-V characteristic curves of the 1D
			resistor simulated with decreasing $\Delta x$,
			where the DSP1(a) and the DSP2(b) are
			considered, respectively}
		\label{fig:toymodel_IV}
	\end{figure}

	\subsection{RTDs}
	In this section, we study the RTD device with the
	same structure as what was studied in
	\cite{pinaud2002transient}. The skeleton of the RTD is shown in \figref{fig:RTD}.
	\begin{figure}[htbp]
		\centering
		\includegraphics[width=\textwidth]{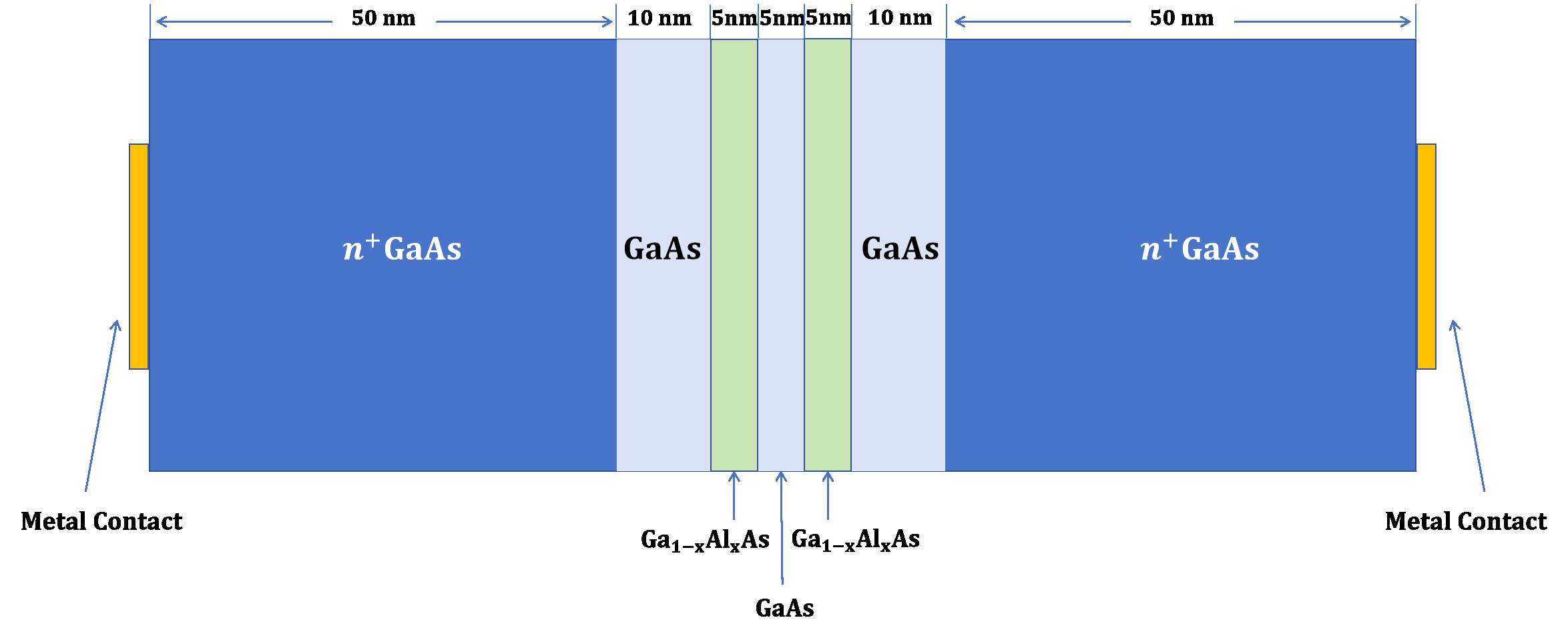}
		\caption{Skeleton of the RTD
		}\label{fig:RTD}
	\end{figure}
	The RTD device is $135(\rm nm)$ long in total,
	and is doped according to
	\begin{equation*}
		N_d(x) = \left\{
		\begin{array}{ll}
			10^{18}\ \rm cm^{-3},\quad & x \in [0,50]\cup [85,135](\rm nm),\\
			5\times 10 ^{15}\ \rm cm^{-3},& x\in (50,85)(\rm nm),\\
		\end{array}
		\right.
	\end{equation*}
	which is smoothed out in all of the simulations.
	The rest of parameters read: $m = 0.067m_0$, where
	$m_0 = 9.1\times 10^{-31}\ (\rm Kg)$,
	$\varepsilon=11.44\varepsilon_0$, where
	$\varepsilon_0 = 8.85\times 10 ^{-12}\ (\rm F\ m^{-1})$,
	and $E_F = 0.0427\ (\rm eV)$.
	
	A double barrier potential with the profile of
	\begin{equation*}
		V_{b}(x) = \left\{
		\begin{array}{ll}
			V_0,\quad & x \in [60,65]\cup [70,75](\rm nm),\\
			0,\quad & x \in  [0, 60)\cup (65,70)\cup (75,135](\rm nm),\\
		\end{array}
		\right.
	\end{equation*}
	is contained within the device,
	where $V_0$ is a prescribed height of the double-barrier
	and is set to $0.3\ \rm (eV)$ in the following experiments.
	The other component of the potential function $V(x)$,
	i.e., $V_s(x)$ is either set in advance(RTD A), or studied
	with a Poisson equation self-consistently(RTD B).
	
	\subsubsection{RTD A}
	\label{sec:subsec:RTDA}
	In this example, $V_s(x)$ is a continuous function
	set in advance,
	which reads
	\begin{equation}
		V_s(x) =     \left\{
		\begin{array}{ll}
			0,\quad & x \in [0,50](\rm nm),\\
			-q_eV_{ds}\dfrac{x-50}{35},\quad & x \in  (50,85)(\rm nm),\\
			-q_eV_{ds},\quad & x \in [85,135](\rm nm).
		\end{array}
		\right.
	\end{equation}
	We smooth $V(x)$ with \eqref{eq:convolution} in order to conduct accuracy order tests.
	In this example, $V_{ds}$ is increased from zero,
	and by $\Delta V_{ds} = 0.02\ \rm (V)$ each time.
	We simulate the I-V characteristic curve
	of RTD A with {\textbf{Algorithm~\ref{iteration_scheme}}},
	and consider the D4TBCs and the aDTBCs
	in applying the algorithm.
	Results about the I-V characteristic I-V curve
	are shown in \figref{fig:linearIV},
	for both discrete boundary conditions, respectively.
	Obviously, current densities solved with both
	of the D4TBCs and the aDTBCs,
	show sharp peaks with similar heights around $0.18\ \rm (V)$.
	Furthermore, all of these I-V characteristic curves
	show negative resistance,
	which corresponds to one of the most typical
	characteristics of quantum devices.
	To test the convergence of both the DS1 and the DS2,
	we use gradually refined mesh grids to simulate
	I-V characteristic curves, where the finest
	grid space is $0.5\ \rm (nm)$.
	Clearly, fast convergence of I-V curves is observed in
	both simulations with the DS1 and the DS2, respectively.
	
	\begin{figure}[htbp]
		\centering  
		\subfigure[DS1]{
			\label{Fig7.sub.a}
			\includegraphics[width=0.45\textwidth]{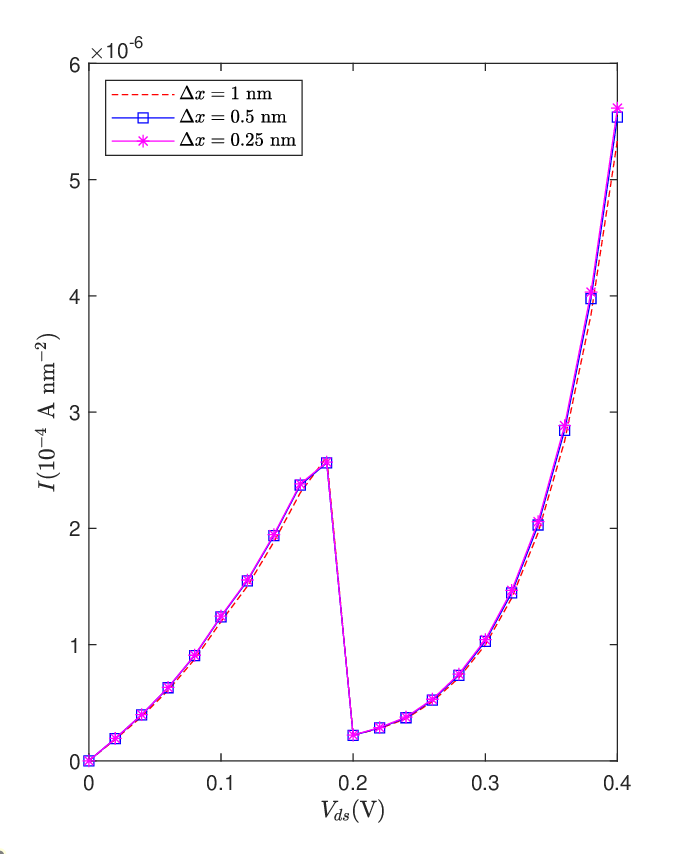}}
		\subfigure[DS2]{
			\label{Fig7.sub.b}
			\includegraphics[width=0.45\textwidth]{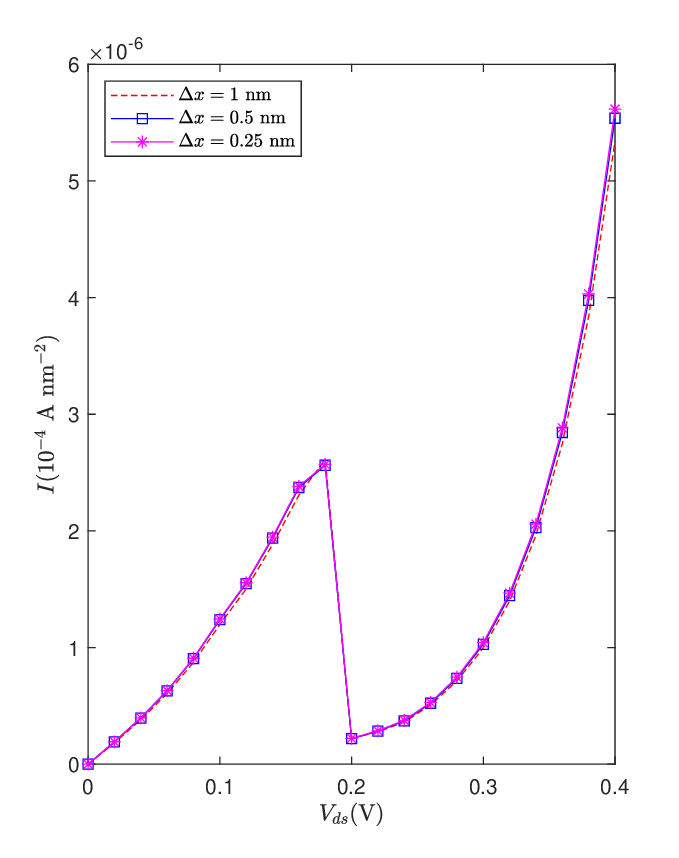}}
		\caption{I-V characteristic curves of RTD A
			with respect to
			DS1(a) and DS2(b), respectively,
			where gradually refined mesh grids are used}
		\label{fig:linearIV}
	\end{figure}

	\subsubsection{RTD B}
	Unlike RTD A studied in
	 Sect.~\ref{sec:subsec:RTDA},
	$V_s(x)$ exerted to RTD B is updated simultaneously
	with $\psi(x;k)$, via solving
	\eqref{Poisson_discretization} equipped with
	\eqref{poisson_dbc} self-consistently, according to
	\textbf{Algorithm~\ref{iteration_scheme}}.
	The temperature of lattice is set to $300\ \rm (K)$,
	$\Delta V_{ds} = 0.02\ \rm (V)$ and $\Delta x = 0.5\ \rm (nm)$.
	I-V characteristic curves corresponding to the
	DSP1 and the DSP2, are shown in \figref{fig:IV}, respectively.
	Obviously, the two I-V characteristic curves
	corresponding to the DSP1 and the DSP2 match very well,
	and they both show peak values around $V_{ds} =0.26\ \rm (V)$.
	In addition, the negative resistance phenomenon is observed
	from both curves.
	\begin{figure}[htbp]
		\centering
		\includegraphics[width=0.8\textwidth]{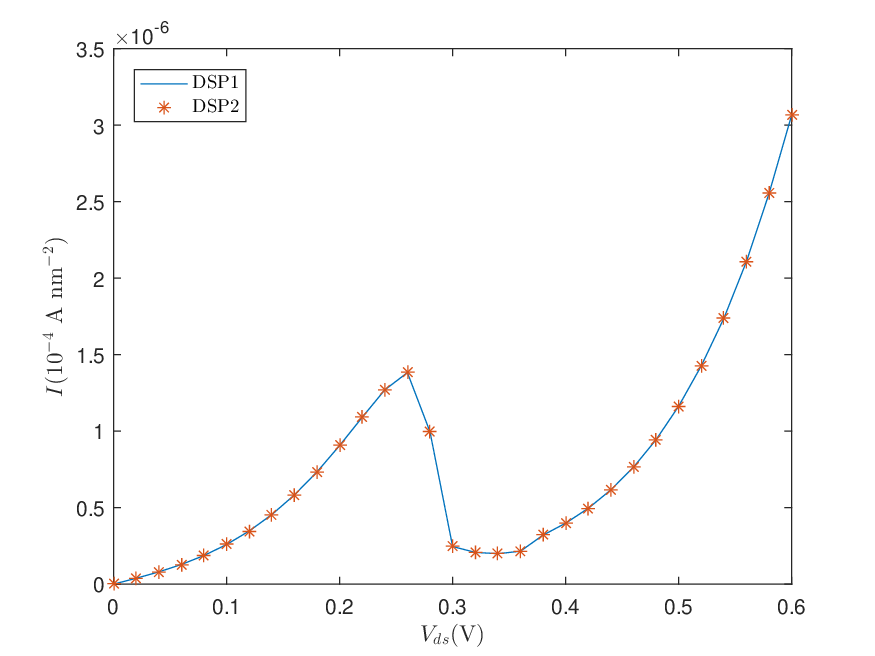}
		\caption{Simulation results about the I-V characteristic
			curve of RTD B, where the DSP1 and the DSP2 are solved
			respectively}\label{fig:IV}
	\end{figure}
	
	In \figref{fig:DenPot0.1}, \figref{fig:DenPot0.26}
	and \figref{fig:DenPot0.28},
	we show stable electron densities and potential functions
	at three different values of $V_{ds}$, which are
	$0.1 \ \rm (V)$, $0.26 \ \rm (V)$ and $0.28\ \rm (V)$, respectively.
	One observes that there are more electrons within the
	quantum well at $0.26\ \rm (V)$, than that at
	$0.1\ \rm (V)$ and $0.28\ \rm (V)$, both.
	This phenomenon is expected, since the current density
	takes its peak value at $0.26\ \rm (V)$ as is
	shown in \figref{fig:IV}.

	\begin{figure}[htbp]
		\centering  
		\subfigure[DSP1]{
			\label{Fig9.sub.a}
			\includegraphics[width=0.45\textwidth]{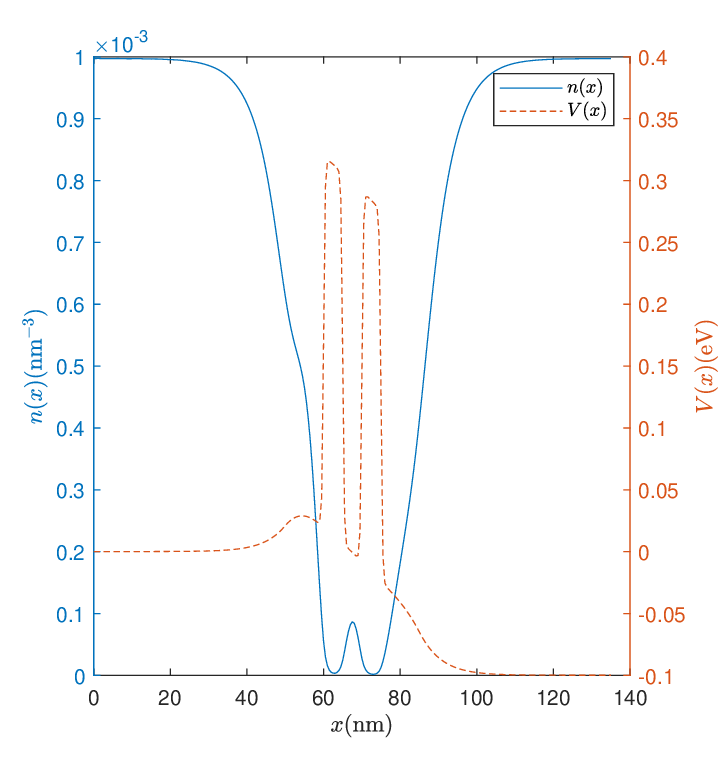}}
		\subfigure[DSP2]{
			\label{Fig9.sub.b}
			\includegraphics[width=0.45\textwidth]{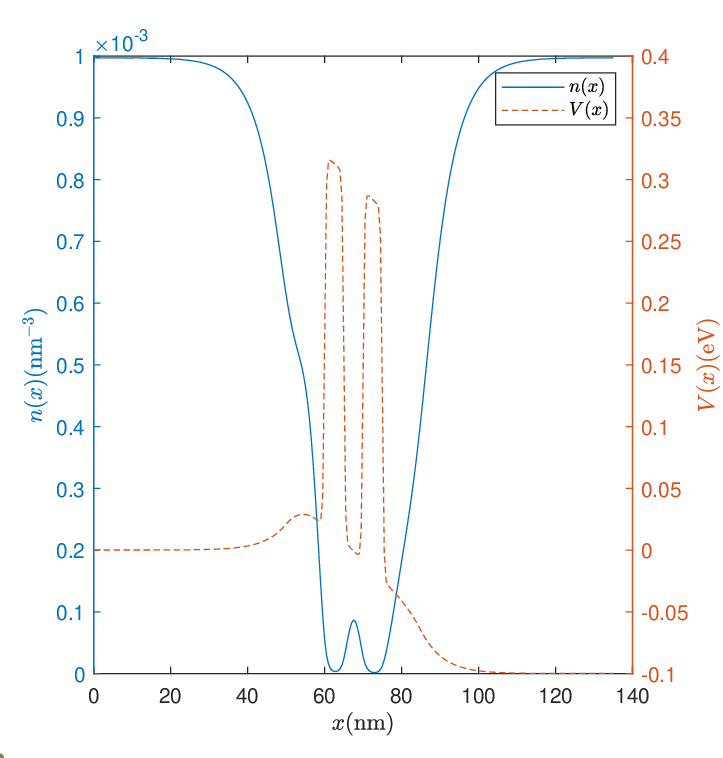}}
		\caption{Electron density $n(x)$ and total potential $V(x)$
			simulated with $V_{ds} = 0.1\ \rm (V)$}
		\label{fig:DenPot0.1}
	\end{figure}
	
	\begin{figure}[htbp]
		\centering  
		\subfigure[DSP1]{
			\label{Fig10.sub.a}
			\includegraphics[width=0.45\textwidth]{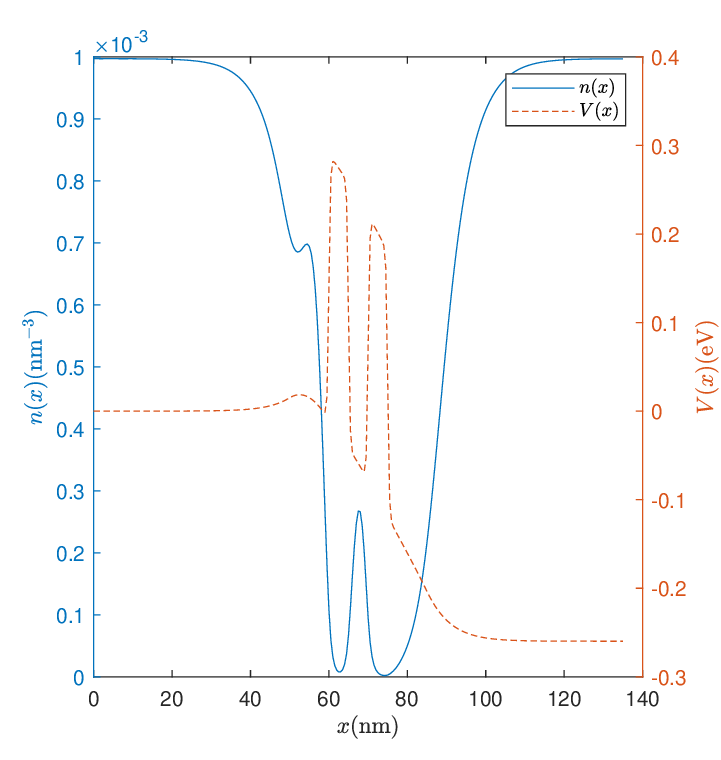}}
		\subfigure[DSP2]{
			\label{Fig10.sub.b}
			\includegraphics[width=0.45\textwidth]{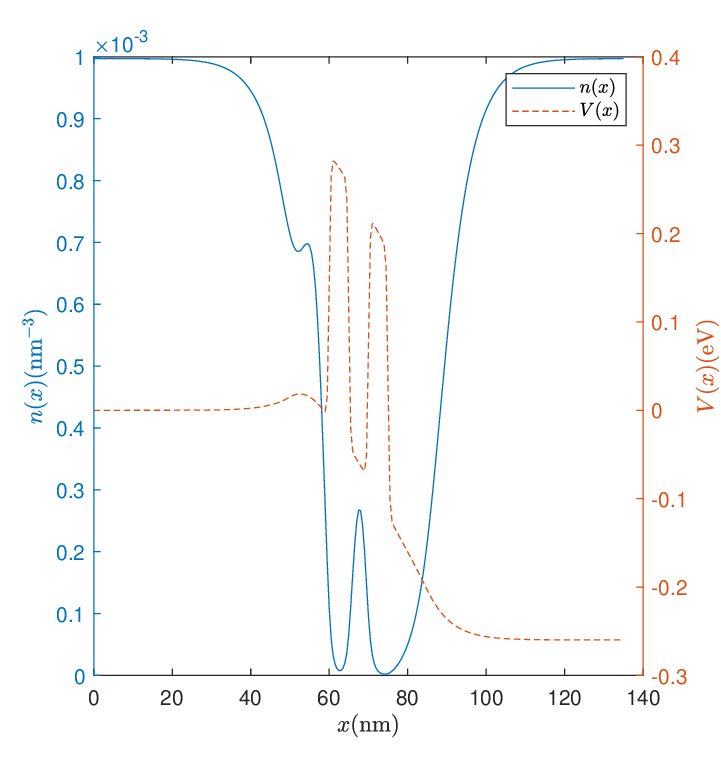}}
		\caption{Electron density $n(x)$ and total potential $V(x)$
			simulated with $V_{ds} = 0.26\ \rm (V)$}
		\label{fig:DenPot0.26}
	\end{figure}
	
	\begin{figure}[htbp]
		\centering  
		\subfigure[DSP1]{
			\label{Fig11.sub.a}
			\includegraphics[width=0.45\textwidth]{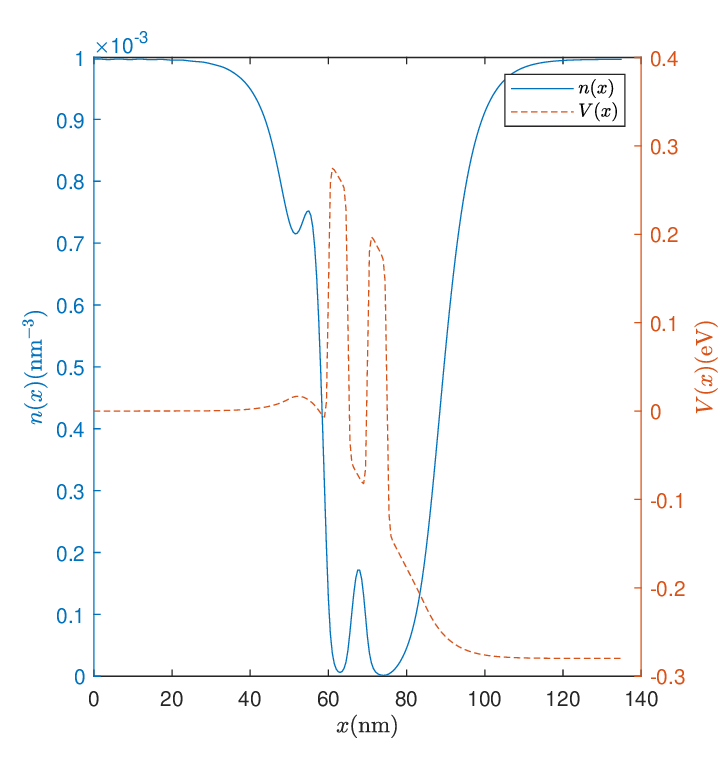}}
		\subfigure[DSP2]{
			\label{Fig11.sub.b}
			\includegraphics[width=0.45\textwidth]{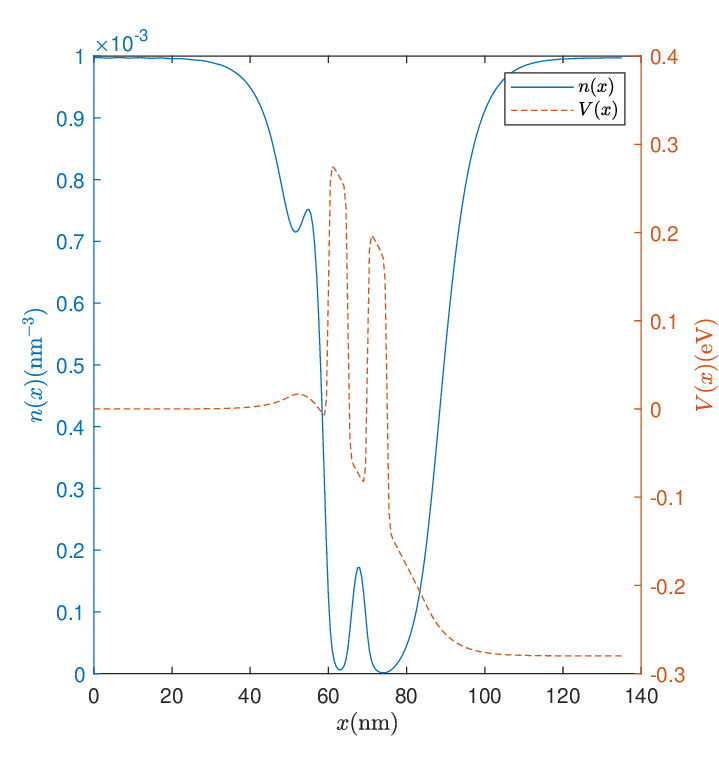}}
		\caption{Electron density $n(x)$ and total potential $V(x)$
			simulated with $V_{ds} = 0.28\ \rm (V)$}
		\label{fig:DenPot0.28}
	\end{figure}
	
	At last, we show transmission coefficient of the device,
	where simulation results obtained with the DSP1 and
	the DSP2 are shown separately in \figref{fig:TE}.
	On the one hand, curves simulated with the DSP1 and
	the DSP2 are very similar, which indicates the accuracy
	of both models.
	On the other hand, the transmission coefficient exhibits a series
	of resonant peaks which could be explained according to \cite{2020A}.
	When $E$ is close to the resonant levels of the well,
	the wave proceeds through the first barrier
	into the well, while simultaneously a wave of equal
	intensity flows out of the well on the other side
	through the other barrier.
	In addition, peaks of the transmission coefficient
	shift towards the lower energy region, and the peak value
	decreases when $V_{ds}$ increases.
	In fact, the transmission coefficient can be
	approximated by a Lorentzian-type function
	\begin{equation*}
		T(E,V_{ds}) = \dfrac{\sqrt{1-\left(
				\frac{q_e V_{ds}}{2E_n}\right)^2}}
		{1+\left(\frac{2(E-E_n)}{\Delta E}\right)^2},
	\end{equation*}
	where $E_n$ is the energy of the resonant level
	and $\Delta E$ denotes the total broadening
	of the energy level. Since resonance occurs when
	$E = E_n - \frac{q_e V_{ds}}{2}$, the resonant peaks
	shift when $V_{ds}$ changes.
	
	\begin{figure}[htbp]
		\centering  
		\subfigure[DSP1]{
			\label{Fig12.sub.a}
			\includegraphics[width=0.45\textwidth]{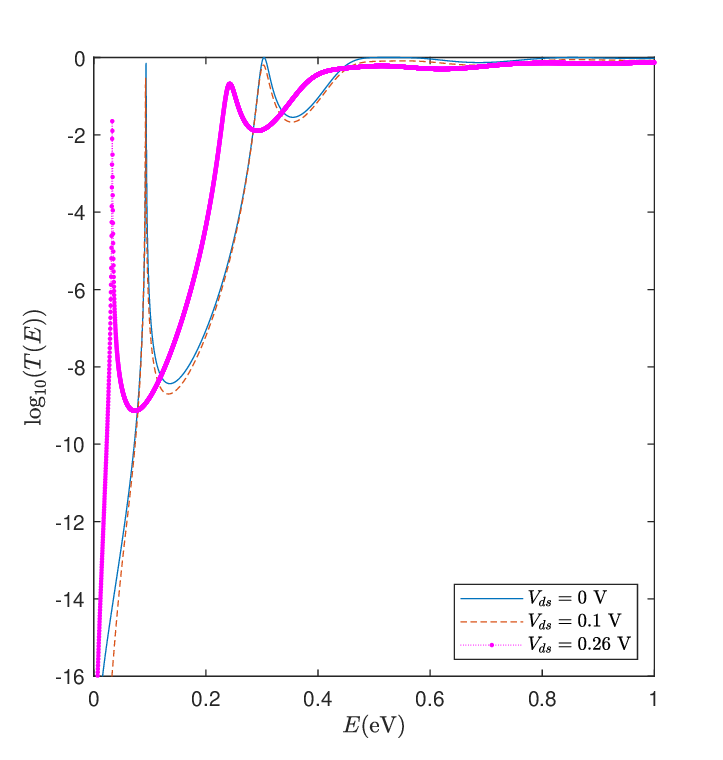}}
		\subfigure[DSP2]{
			\label{Fig12.sub.b}
			\includegraphics[width=0.45\textwidth]{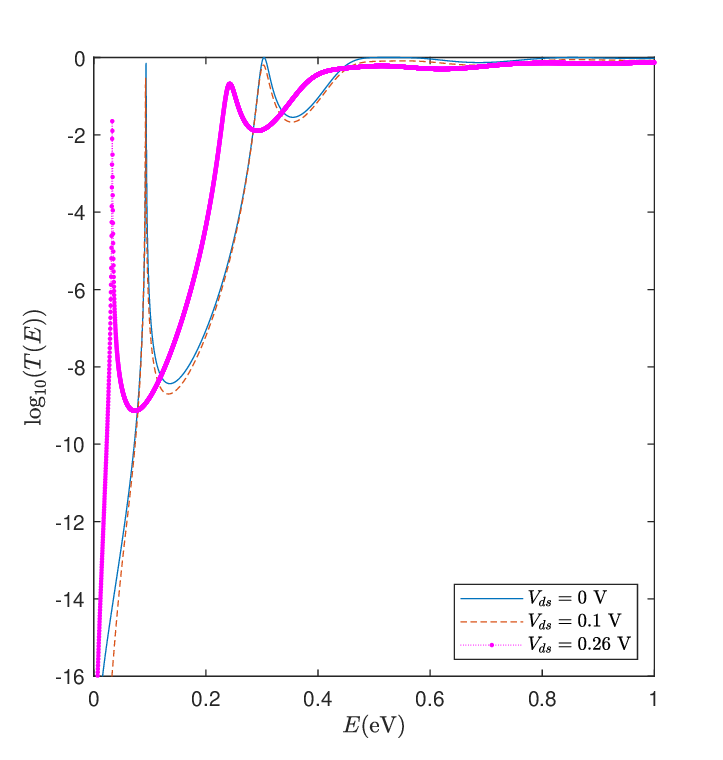}}
		\caption{Transmission coefficient in logarithmic scale for three different biases}
		\label{fig:TE}
	\end{figure}
	
	\section{Conclusion}
	In this paper, we introduce two optimal discretization schemes for
	the TBCs of the 1D Schr{\"o}dinger equation, i.e.,
	the D4TBCs and the aDTBCs.
	The D4TBCs could essentially avoid spurious oscillation
	in numerical solutions when the potential vanishes, and
	possesses the same accuracy order with the discretization
	scheme applied to discretize the 1D Schr{\"o}dinger equation.
	However, the D4TBCs can not be generalized to arbitrarily high
	order since an algebraic equation should be solved explicitly,
	which is impossible because the order of the algebraic equation
	increases as the accuracy order increases.
	To overcome that, we further propose a framework of aDTBCs, which
	is accurate in discretizing the TBCs, and could easily
	be generalized to arbitrarily high order.
	Involving the D4TBCs and the aDTBCs, two globally
	fourth order compact finite difference schemes
	are introduced for discretizing the 1D Schr{\"o}dinger-Poisson
	problem, and an algorithm is proposed to simulate semiconductor
	devices.
	Numerical simulations are carried out for a classic
	resistor and two RTDs, and simulation results
	successfully verify the accuracy orders of both schemes
	and show characteristics of these devices.

\section*{Declarations}

\begin{Funding}
	This work was supported in part by the NSFC (Grant Numbers: 12171035, 12131001) and the Natural Science Foundation of Guangdong Province of China (Grant Number: 2024A1515010356).
\end{Funding}

\begin{data}
	The datasets generated during and/or analysed during the current study are available in the figshare repository, \href{https://figshare.com/articles/dataset/1DSchrodinger\_TBC/25765377}{https://figshare.com/articles/dataset/1DSchrodinger\_TBC/25765377}.
\end{data}

\begin{Conflict}
	The authors declare that they have no known competing financial 
	interests or personal relationships that could have appeared 
	to influence the work reported in this paper.
\end{Conflict}


\bibliographystyle{spbasic}
\bibliography{ref.bib}   

\end{document}